\documentclass[a4paper,11pt]{amsart}

\usepackage{amssymb}

\usepackage{amssymb}
\usepackage{amscd}
\usepackage[dvips]{graphicx}
\usepackage{amsfonts}
\usepackage{amsthm}
\usepackage{amsmath}
\usepackage{amssymb}

\title[Steady periodic waves with constant vorticity]{\bf Steady periodic water waves with constant vorticity:
 regularity and local bifurcation }
\author{Adrian Constantin and Eugen Varvaruca}\date{}

\newtheorem{theorem}{Theorem}[section]
\newtheorem{lemma}[theorem]{Lemma}

\newtheorem{remark}[theorem]{Remark}

\numberwithin{equation}{section}

\newcommand{\Om}{\Omega}

\newcommand{\om}{\omega}
\newcommand{\la}{\lambda}

\newcommand{\be}{\begin{equation}}
\newcommand{\ee}{\end{equation}}

\newcommand{\mcs}{{\mathcal S}}

\newcommand{\mcb}{{\mathcal B}}

\newcommand{\mcg}{{\mathcal G}}
\newcommand{\mcl}{{\mathcal L}}
\newcommand{\mcr}{\mathcal{R}}

\newcommand{\mcx}{{\mathbb X}}
\newcommand{\mcy}{{\mathbb Y}}

\newcommand{\mck}{{\mathcal K}}
\newcommand{\pa}{\partial}

\newcommand{\ze}{\zeta}

\newcommand{\bdc}{\mathbb{C}}
\newcommand{\bdd}{\mathbb{D}}

\newcommand{\bdr}{\mathbb{R}}
\newcommand{\bdz}{\mathbb{Z}}

\newcommand{\mcc}{\mathcal{C}}
\newcommand{\al}{\alpha}
\newcommand{\veps}{\varepsilon}

\newcommand{\bese}{\begin{subequations}}
\newcommand{\ese}{\end{subequations}}
\newcommand{\non}{\nonumber}

\begin{document}
\begin{abstract}
This paper studies periodic traveling gravity waves at the free
surface of water in a flow of constant vorticity over a flat bed.
Using conformal mappings the free-boundary problem is transformed
into a quasilinear pseudodifferential equation for a periodic
function of one variable. The new formulation leads to a regularity
result and, by use of bifurcation theory, to the existence of waves
of small amplitude even in the presence of stagnation points in the
flow.
\end{abstract}

\maketitle

{\small \noindent {\it Keywords and Phrases}: water waves,
vorticity, conformal mapping, Hilbert transform.

}

\section{Introduction}

The problem of spatially periodic traveling gravity water waves with
constant vorticity $\gamma$ on a flow of finite depth can be
 formulated as the free-boundary problem of finding
\begin{itemize}
\item a domain $\Om$ in $(X,Y)$-plane, whose boundary consists of the
real axis
 \begin{subequations}\label{apb}
 \be\mcb=\{(X,0):X\in\bdr\},\ee representing the flat
impermeable water bed, and an a priori unknown curve expressed in
parametric form as \be\mcs=\{(u(s),v(s)):s\in\bdr\},\label{spar}\ee
with \be u(s+L)=u(s)+L,\quad v(s+L)=v(s)\quad\text{for all
}s\in\bdr,\label{lper}\ee representing the free surface of the
water, which is $L$-periodic in the horizontal direction;
\item a function $(X,Y)\mapsto\psi(X,Y)$ which is $L$-periodic in $X$ throughout
$\Om$, representing the stream function giving the velocity field
$(\psi_Y,-\psi_X)$ in a frame moving at the constant wave speed,
which satisfies the following equations and boundary conditions:
\begin{align}
& \Delta\psi=-\gamma\quad\text{in }\Om,\label{ap0}\\
& \psi=-m\quad \text{on }\mcb,\label{ap1}\\&\psi=0\quad \text{on
}\mcs, \label{ap2}\\&\vert\nabla\psi\vert^{2}+2gY=Q \quad\text{on }
\mcs.\label{ap3}
 \end{align}
\end{subequations}
\end{itemize}
 Here $g$ is the gravitational constant of acceleration, the
constant $m$ is the relative mass flux, while the constant $Q$ is
related to the hydraulic head (see the discussion in \cite{CS}). The
level sets of $\psi$ are the streamlines, while a point where the
gradient of $\psi$ vanishes is called a \emph{stagnation point}.

We are interested in the existence, regularity and geometric
properties  of solutions of (\ref{apb}). The physical motivation for
our study lies in the desire to understand wave-current
interactions, since nonzero vorticity is the hallmark of a
non-uniform underlying current \cite{CE, CS}. Even in the simplest
case of a linear current, which corresponds to a constant vorticity
$\gamma$ in (\ref{apb}), numerical studies \cite{DP, KSt, KSt2, OS}
point to the existence of large-amplitude waves with overhanging
profiles and critical layers (regions in the fluid bounded by closed
streamlines and containing stagnation points). This is in striking
contrast to the extensively studied irrotational case
\cite{BT,ST,T}, for which any wave profile is necessarily the graph
of a function and there can be no critical layers in the fluid
\cite{T1,V,V2}. The assumption of an underlying current with nonzero
constant vorticity (or linear shear) provides us with the simplest
case of a flow that is not irrotational and is attractive for
analytical tractability. However, this setting is not a mere
mathematical convenience, being physically relevant. Indeed, on
areas of the continental shelf and in many coastal inlets the most
significant currents are the tides and tidal flows are the most
regular and predictable currents \cite{Jo, Sw}, considered to be
two-dimensional flows of constant vorticity: negative constant
vorticity is generated when the flow of the current is directed
towards the shore (flood) and positive constant vorticity when the
current is directed back out to sea (ebb) \cite{DP, Ya}.

This paper addresses some of  the mathematical challenges posed by
the study of solutions of (\ref{apb}) with overhanging profiles or
critical layers. In all previous studies of existence of solutions
of (\ref{apb}), equivalent formulations of the problem over a fixed
domain are essential. Bifurcation theory is then the most successful
tool used for existence results. For irrotational flows ($\gamma=0$,
representing uniform flows without vorticity), a hodograph
transformation maps the unknown domain occupied by the water into a
fixed horizontal strip in a complex plane where the variable is the
complex potential of the fluid flow. In these coordinates, the
classical approach is to reformulate (\ref{apb}) as a nonlinear
singular integral equation for a function of one variable which
gives the angle of inclination between the tangent to the free
surface and the horizontal \cite{AT, KN, T}. A more recent approach
\cite{BDT1, BDT2, ST}  (for the analogous problem of waves of
infinite depth) uses a reformulation of the problem as a
pseudodifferential equation for a function of one variable which
gives the elevation of the free surface when the fluid domain is the
conformal image of a strip \cite{Ba}.

In this paper we show that the second approach can be extended to
the case of any constant vorticity $\gamma$. The basic idea is to
seek the fluid domain as the conformal image of a strip, regardless
of the unavailability of a complex potential. The new formulation is
presented in Section 2. It imposes no restriction on the geometry of
the free surface or on the streamline pattern, and in particular it
can handle overhanging profiles and critical layers. In order to
avoid distracting technicalities, we assume at the outset a modest
amount of regularity of the free surface, namely $C^{1,\alpha}$ for
some $\alpha\in (0,1)$, though arguing along the lines of \cite{ST,
V2} one could verify the equivalence of the formulations also in the
framework of weak solutions \cite{V1}. As the new formulation
involves the periodic Dirichlet-Neumann operator and the periodic
Hilbert transform associated to a strip,  Section 3 studies some
useful properties of these operators. Section 4 is devoted to
regularity of solutions of the new equation, showing in particular
that any $C^{1, \alpha}$ free surface with no stagnation points is
actually $C^\infty$. In Section 5 we prove the existence of waves of
small amplitude using bifurcation theory. The analysis reveals the
existence of such waves with a critical layer, a fact first proved
rigorously only very recently in \cite{W}, though anticipated by a
formal argument of Kelvin \cite{Ke}, who presented the famous cat's
eye picture of the streamline pattern.

Before proceeding with the bulk of the paper we would like to
comment upon the relation to some recent investigations on water
waves with vorticity. For general vorticity but under the assumption
that $\psi$ is monotone in the vertical direction, a hodograph
transform  $(X,Y)\mapsto (X,\psi)$ was used in \cite{CS} to
transform (\ref{apb}) into a nonlinear boundary value problem for a
quasilinear elliptic equation in a strip, which permits the
investigation of waves of small and large amplitude (see also
\cite{CEW, CS3, V3, V1}). Thus the formulation in \cite{CS} cannot
describe neither critical layers nor overhanging profiles. The
recent paper \cite{W} deals with waves of small amplitude in a flow
of constant vorticity which could contain stagnation points. The
approach of \cite{W} cannot handle overhanging profiles. For a free
surface of the form $\mcs=\{(X,\eta(X)):X\in\bdr\}$, it reformulates
(\ref{apb}) as a nonlinear boundary problem for an elliptic equation
in a strip by performing the flattening change of variables $(X,Y)
\mapsto (X,\yen)=\Big(X,\displaystyle\frac{Y}{\eta(X)} \Big)$. The
free surface is thus transformed into the horizontal line $\yen=1$
and, in the new variables, one has to solve a system of equations
for a function of two variables and one of one variable. Our
analysis of small amplitude waves is much simpler than that in
 \cite{W}. The new formulation, which has an elegant structure,
  opens up the possibility of using global bifurcation theory to
 prove the existence of waves of large
 amplitude with critical layers and/or overhanging profiles, a prospect which
 seems daunting for the formulation in \cite{W}. In addition, our
 approach has the advantage that in the reformulation we do not
depart from harmonic function theory, and the fine structural
properties associated with harmonic functions continue to be
relevant. Such properties made it possible to describe the particle
motion beneath irrotational waves \cite{C, CS2} and the exploration
of this aspect for flows with constant vorticity is now plausible.

\section{Reformulation of the free-boundary problem}

In this section we present the reformulation of the free-boundary
problem (\ref{apb}) as the quasilinear pseudodifferential equation
(\ref{maineq}) for a periodic function of one variable. This
involves the periodic Dirichlet-Neumann operator and the periodic
Hilbert transform for a strip, so we start with a discussion of the
definitions and basic properties of these operators.

For any integer $p \ge 0$ and $\alpha\in (0,1)$ we denote by
$C^{p,\alpha}$ the standard space of functions whose partial
derivatives up to order $p$ are H\"{o}lder continuous with exponent
$\alpha$ over their domain of definition. By
$C^{p,\al}_{\textnormal{loc}}$ we denote the set of functions of
class $C^{p,\alpha}$ over any compact subset of their domain of
definition. We will need the following local version of Privalov's
Theorem \cite[Chapter V, \S E, p.\ 100]{Ko} (see \cite[Lemma
2.2]{EV} for the proof of a very similar result). A more refined,
global version of Privalov's Theorem, related to but different from
that in \cite{Ko}, will be proved in Appendix B.

\begin{lemma}\label{lreg} For any $t>0$,
we denote $\bdd^+_t=\{(x,y)\in\bdr^2:\, x^2+y^2<t^2,\,y>0\}$. Let
$r>0$ and $Z+iW$ be a holomorphic function in $\bdd^+_r$, where $Z$
and $W$ are real-valued functions. Suppose that $W$ is continuous in
$\bdd^+_r\cup \{(x,0):x\in(-r,r)\}$, and let
\[ w(x)=W(x,0),\qquad x\in(-r,r).\]
 If $w\in C^{p,\al}_{\textnormal{loc}}((-r,r))$ for some integer
$p\geq 0$ and $\alpha\in (0,1)$, then $W$ and $Z$ are of class
$C^{p,\alpha}$ in the closure of $\bdd_t^+(x_0)$ for every $t\in
(0,r)$.
\end{lemma}

For any $d>0$, let $\mcr_d$ be the strip
\[\mcr_d=\{(x,y)\in\bdr^2:-d<y<0\}.\] For any integer $p \ge 0$ and $\alpha\in
(0,1)$ we denote by $C^{p,\alpha}_{2\pi}$ the space of functions of
one real variable which are $2\pi$-periodic and of class
$C^{p,\alpha}$. For any $w\in C^{p,\alpha}_{2\pi}$, let $W\in
C^{p,\alpha}(\overline \mcr_d)$ be the unique solution of
\begin{align}\label{wdn}
&\Delta W=0 \qquad\text{in }\mcr_d,\\
&W(x,-d)=0,\qquad x\in\bdr,\non\\ &W(x,0)=w(x),\qquad
x\in\bdr\non.\end{align} The function $(x,y)\mapsto W(x,y)$ is
$2\pi$-periodic in $x$ throughout $\mcr_d$. For $p\ge 1$ integer, we
define $\mcg_d(w)$ by
\[\mathcal{G}_d(w)(x)=W_y(x,0),\qquad x\in\bdr.\]
(Throughout the paper, the subscripts $x$, $y$, $X$, $Y$ denote
partial derivatives, and are the only subscripts with this meaning.)
Then $\mathcal{G}_d(w)\in C^{p-1,\alpha}_{2\pi}$. The mapping
$w\mapsto \mcg_d(w)$ is called the \emph{periodic Dirichlet-Neumann
operator for a strip}, and is a (bounded) linear operator from
$C^{p,\alpha}_{2\pi}$ to $ C^{p-1,\alpha}_{2\pi}$. Note in
particular that if $w$ is a constant function taking the value $c$,
then
\[ W(x,y)= \frac{c}{d}(y+d), \quad(x,y)\in\mcr_d,\]
and hence \be \mcg_d(c)=c/d.\label{dnc}\ee Suppose now that $Z$ is a
harmonic function in $\mcr_d$, uniquely determined up to a constant,
such that $Z+iW$ is holomorphic in $\mcr_d$. By Lemma \ref{lreg},
$Z\in C^{p,\alpha}(\overline \mcr_d)$. Note that the function
\[(x,y)\mapsto Z(x+2\pi,y)-Z(x,y)\]
is constant in $\mcr_d$, as can be seen by taking its partial
derivatives and using the Cauchy-Riemann equations. Let $K\in\bdr$
be such that \be Z(x+2\pi,y)-Z(x,y)=K,\qquad
(x,y)\in\mcr_d.\label{perz}\ee For every $x\in\bdr$ and $y\in (-d,
0)$, it follows using the Cauchy-Riemann equations that
\begin{align}\label{ka}
\frac{d}{dy}\int_{x}^{x+2\pi}W(t,y)\,dt = \int_{x}^{x+2\pi}W_y(t,y)\,dt\\
= \int_{x}^{x+2\pi}Z_x(t,y)\,dt=Z(x+2\pi,y)-Z(x,y)=K.\non
\end{align}
Integrating the above relation from $y=-d$ to $y=0$, and taking into
account the boundary conditions satisfied by $W$ in (\ref{wdn}), it
follows that \be K=\frac{2\pi[w]}{d},\label{cap}\ee where $[w]$
denotes the average of $w$ over one period.
 Hence $K=0$ if and only if $[w]=0$. For any integer $p\geq 0$ and $\alpha\in
(0,1)$, let us denote by $C^{p,\alpha}_{2\pi,\circ}$ the class of
functions in $C^{p,\alpha}_{2\pi}$ which have zero mean over one
period.

Let $w\in C^{p,\alpha}_{2\pi,\circ}$. Then the function
$(x,y)\mapsto Z(x,y)$ is $2\pi$-periodic in $x$ throughout $\mcr_d$.
We normalize the constant in the definition of $Z$ by the
requirement that $x\mapsto Z(x,0)$, which is a $2\pi$-periodic
function, has zero mean over one period. We define $\mcc_d(w)$ by
\[\mcc_d (w)(x)=Z(x,0),\qquad x\in\bdr.\]
 The mapping $w\mapsto
\mcc_d (w)$ is called the \emph{periodic Hilbert transform for a
strip}, and is a (bounded) linear operator from
$C^{p,\alpha}_{2\pi,\circ}$ into itself. It is easy to see from the
definitions and the Cauchy-Riemann equations that, when $w\in
C^{p,\alpha}_{2\pi,\circ}$ for $p\geq 1$ integer,
\[ \mcg_d(w)=(\mcc_d (w))'=\mcc_d(w').\]
 A more
detailed account of the properties of the operators $\mcg_d$ and
$\mcc_d$ is given in Section 3. Note that the space of
$2\pi$-periodic harmonic/holomorphic functions in a strip can be
identified with the space of harmonic/holomorphic functions in an
annulus, by means of an explicit conformal mapping, and thus the
operators $\mcg_d$ and $\mcc_d$ can be equivalently defined in the
setting of an annulus.

Suppose now that $w\in C^{p,\alpha}_{2\pi}$. By (\ref{dnc}),
(\ref{perz}) and (\ref{cap}), one can write \[
Z+iW=\frac{[w]}{d}(x+i(y+d))+Z_0+iW_0\qquad\text{in }\mcr_d,\] where
$Z_0, W_0\in C^{p,\alpha}(\overline\mcr_d)$ are $2\pi$-periodic in
$x$ throughout $\mcr_d$ and
\begin{align}\label{wdn0}
&\Delta W_0=0 \qquad\text{in }\mcr_d,\\
&W_0(x,-d)=0,\qquad x\in\bdr,\non\\ &W_0(x,0)=w(x)-[w],\qquad
x\in\bdr\non.\end{align} We normalize the constant in the definition
of $Z$ by the requirement that $x\mapsto Z_0(x,0)$, which is a
$2\pi$-periodic function, has zero mean over one period. It follows
that \be Z(x,0)=\frac{[w]}{d}x+\mcc_d(w-[w]).\label{parz}\ee It also
follows, using the Cauchy-Riemann equations, that, when $p\ge 1$,
\be\mcg_d(w)=\frac{[w]}{d}+(\mcc_d(w-[w]))'=\frac{[w]}{d}+\mcc_d(w').\label{cone}\ee

Throughout the paper we are interested in solutions $(\Om,\psi)$ of
the water-wave problem (\ref{apb}) of class $C^{1,\alpha}$, for some
$\alpha\in (0,1)$, by which which we mean that $\mcs$ has a
parametrization (\ref{spar}) with $u,\, v$ functions of class
$C^{1,\alpha}$, such that (\ref{lper}) holds and
\[u'(s)^2+v'(s)^2\neq 0\quad\text{for all }s\in\bdr,\]
while $\psi\in C^\infty(\Om)\cap C^{1,\alpha}(\overline\Om)$. The
main result of this section is that the free-boundary problem
(\ref{apb}) is then equivalent to the problem of finding a positive
number $h$ and a function $v\in C^{1,\alpha}_{2\pi}$ which satisfy
the following:
\begin{subequations}\label{maineq}
\begin{align}&\left\{\frac{m}{kh} +\gamma \Big(
\mathcal{G}_{kh}( v^2/2)- v \mathcal{G}_{kh}(v)\Big)\right\}^2=
(Q-2gv)\,\Big(v'^2 + \mcg_{kh}(v)^2\Big),\label{m0}\\&
[v]=h,\label{m1}\\
&v(x)>0\quad\text{for all }x\in\bdr,\label{pos}\\ &\text{the mapping
$x\mapsto \left(\frac{x}{k}+\mcc_{kh}(v-h)(x), v(x)\right)$ is
injective on $\bdr$},\label{m2}\\& v'(x)^2+\mcg_{kh}(v)(x)^2\neq
0\quad\text{for all }x\in\bdr,\label{m3}
\end{align}
\end{subequations}
where \be L=2\pi/k,\qquad k>0.\label{lk}\ee (Thus $2\pi$-periodicity
in $x$ becomes $L$-periodicity in $s=x/k$.) More precisely, we have
the following.

\begin{theorem}\label{refor} Let $(\Om,\psi)$ be a solution of {\rm (\ref{apb})} of class $C^{1,\alpha}$.
 Then there exist a positive number $h$, a function $v\in
C^{1,\alpha}_{2\pi}$ and a constant $a\in\bdr$ such that {\rm
(\ref{maineq})} holds and \be
\mcs=\left\{\left(a+\frac{x}{k}+\mcc_{kh}(v-h)(x), v(x)\right):
x\in\bdr\right\}.\label{paras}\ee Conversely, let $h>0$ and  $v\in
C^{1,\alpha}_{2\pi}$ be such that {\rm (\ref{maineq})} holds, and
let $a\in\bdr$ be arbitrary. Let $\mcs$ be defined by {\rm
(\ref{paras})}, and $\Om$ be the domain whose boundary consists of
$\mcs$ and the real axis $\mcb$. Then there exists a function $\psi$
in $\Om$ such that $(\Om,\psi)$ is a solution of {\rm (\ref{apb})}
of class $C^{1,\alpha}$.
\end{theorem}

The presence of the constant $a$ in (\ref{paras}) is due to the
invariance of problem (\ref{apb}) to horizontal translations. We now
explain the meaning of the positive constant $h$ which arises in
(\ref{maineq}). A domain $\Om$ contained in the upper half of the
$(X,Y)$-plane is called an \emph{$L$-periodic strip-like domain} if
its boundary consists of the real axis $\mcb$ and a curve $\mcs$
described in parametric form by (\ref{spar}) such that (\ref{lper})
holds. For any such domain, we define its \emph{conformal mean
depth} as the unique positive number $h$ such that there exists a
conformal mapping $\tilde U+i\tilde V$ from the strip $\mcr_h$ onto
$\Om$, which admits an extension as a homeomorphism between the
closures of these domains, with $\{(x,0):x\in\bdr\}$ being mapped
onto $\mcs$ and $\{(x,-h):x\in\bdr\}$ being mapped onto $\mcb$, and
 such that \be \tilde U(x+L,y)=\tilde U(x,y)+L,\quad
\tilde V(x+L,y)= \tilde V(x,y),\qquad
(x,y)\in\mcr_h.\label{xcoum}\ee The existence of a unique number $h$
with this property is proved in Appendix A. (The notion of conformal
mean depth is reminiscent of that of the conformal modulus for
doubly connected domains, defined as the unique number $R\in (0,1)$
such that the domain is conformally equivalent to the annulus
$\{(x,y)\in\bdr^2: R^2<x^2+y^2<1\}$, see \cite{Ne, Ru}.) Note that,
for any horizontal strip, its conformal mean depth coincides with
its usual depth (height). Also, as we show in Appendix A, the
conformal mean depth coincides with the `mean depth' used in
\cite{AT} in the study of irrotational flows. The proof of Theorem
\ref{refor} reveals that $h$ in (\ref{maineq}) is exactly the
conformal mean depth of $\Om$.

\begin{proof}[Proof of Theorem \ref{refor}.] Let $(\Om,\psi)$ be a
solution of (\ref{apb}) of class $C^{1,\alpha}$.
 Let $h$ be the conformal mean depth of
$\Om$, and let $\tilde U+i\tilde V$ be the associated conformal
mapping. As we show in Appendix A, it is a consequence of the
Kellogg-Warschawski Theorem that $\tilde U,\,\tilde V \in
C^{1,\alpha}(\overline{\mcr_h})$ and \[\quad\tilde U_x^2(x,0)+\tilde
V_x^2(x,0)\neq 0\qquad\text{for all }x\in\bdr.\]  Consider the
mapping $U+iV:\mcr_{kh}\to\Om$ given by \be U(x,y)=\tilde U(x/k,
y/k),\qquad V(x,y)=\tilde V(x/k, y/k),\qquad
(x,y)\in\mcr_{kh},\label{uv} \ee where $k$ is given by (\ref{lk}).
Then $U+iV$ is a conformal mapping from $\mcr_{kh}$ onto $\Om$,
which admits an extension as a homeomorphism between the closures of
these domains, with $\{(x,0):x\in\bdr\}$ being mapped onto $\mcs$
and $\{(x,-kh):x\in\bdr\}$ being mapped onto $\mcb$, and
 such that \be U(x+2\pi,y)= U(x,y)+\frac{2\pi}{k},\quad
 V(x+2\pi,y)= V(x,y),\qquad
(x,y)\in\mcr_{kh}.\label{xcou}\ee Moreover, $U,\,V \in
C^{1,\alpha}(\overline{\mcr_h})$ and \be U_x^2(x,0)+ V_x^2(x,0)\neq
0\qquad\text{for all }x\in\bdr,\label{mapbd}\ee while the conformal
mapping properties of $U+iV$ imply that \be
V(x,-h)=0,\label{cou1}\ee \be \text{the mapping $x\mapsto (U(x,0),
V(x,0))$ is injective on $\bdr$},\label{inj}\ee
 \be \mcs=\{(U(x,0),
V(x,0)):x\in\bdr\}.\label{S}\ee Let \be v(x)=V(x,0) \quad\text{for
all }x\in\bdr.\label{vv}\ee Then $v\in C^{1,\alpha}_{2\pi}$. Since
$V$ is harmonic in $\mcr_{kh}$ and satisfies (\ref{cou1}) and
(\ref{vv}), we deduce, by comparing (\ref{xcou}) with (\ref{perz}),
(\ref{ka}) and (\ref{cap}),  that \be [v]=h.\label{averv}\ee  It
follows, by using (\ref{parz}), that (\ref{inj}) and (\ref{S}) can
be rewritten as \be \text{the mapping $x\mapsto
\left(\frac{x}{k}+\mcc_{kh}(v-h)(x), v(x)\right)$ is injective on
$\bdr$},\label{inj1}\ee \be
\mcs=\left\{\left(a+\frac{x}{k}+\mcc_{kh}(v-h)(x),
v(x)\right):x\in\bdr\right\},\label{S1}\ee for some $a\in\bdr$.
Also, using the Cauchy-Riemann equations one can rewrite
(\ref{mapbd}) as \be v'(x)^2+\mcg_{kh}(v)(x)^2\neq 0\qquad\text{for
all }x\in\bdr\label{mapbd1}.\ee Observe also that, since $\mcs$ is
contained in the upper half-plane, it follows that \be
v(x)>0\quad\text{for all }x\in\bdr.\label{posi}\ee

 Suppose now that $\psi$ satisfies (\ref{apb}), and let
$\xi:\mcr_{kh}\to\bdr$ be given by \be \xi(x,y)=\psi(U(x,y),
V(x,y)),\qquad (x,y)\in\mcr_{kh}.\label{xi}\ee
 Note that (\ref{ap0}) can be rewritten as
\[
(X,Y)\mapsto\psi(X,Y)+\frac{\gamma}{2} Y^2\quad\text{is a harmonic
function in $\Om$}.\] Hence, since harmonic functions are invariant
under conformal mappings, it follows that
\begin{subequations}\label{mapb}
\be \xi+\frac{\gamma}{2}V^2\quad\text{is harmonic in
$\mcr_{kh}$}.\label{map0}\ee It is immediate from (\ref{ap1}) and
(\ref{ap2}) that
\begin{align}&\xi(x,-kh)=-m, \qquad x\in\bdr,\label{map1}\\ &\xi(x,0)=0,\qquad
x\in\bdr.\label{map2}\end{align} By using the chain rule and the
Cauchy-Riemann equations, we easily obtain that
\[{\xi_x^2+\xi_y^2}= (\psi_X^2(U,V)+\psi_Y^2(U,V))(V_x^2+V_y^2)\quad\text{in }\overline{\mcr_{kh}}.\]
It follows from (\ref{ap3}), taking into account (\ref{mapbd}), that
\be \xi_x^2+\xi_y^2=(Q-2gV)(V_x^2+V_y^2)\quad\text{at $(x,0)$ for
all $x\in\bdr$.}\label{map3}\ee
\end{subequations}
Let $\zeta:\mcr_{kh}\to\bdr$ be given by
\begin{equation}\label{zeta}
\zeta=\xi+m+\frac{\gamma}{2} V^2.
\end{equation}
Then (\ref{mapb}) can be equivalently rewritten as
\begin{subequations}\label{napb}
\begin{align}
& \Delta\ze=0\quad\text{in }\mcr_{kh},\label{nap0}\\&
\ze(x,-kh)=0\label{nap1}\quad\text{for all
$x\in\bdr$},\\&\ze(x,0)=m+\frac{\gamma}{2}v^2(x)\quad\text{for all
$x\in\bdr$},
\label{nap2}\\
 &(\ze_y-\gamma VV_y)^2=(Q-2gV)(V_x^2+V_y^2)
\quad\text{at $(x,0)$ for all $x\in\bdr$}\label{nap3}.
 \end{align}
\end{subequations}

Equation (\ref{napb}) can be conveniently expressed by means of the
Dirichlet-Neumann operator as \be \left\{\frac{m}{kh} +\gamma \Big(
\mathcal{G}_{kh}( v^2/2)- v \mathcal{G}_{kh}(v)\Big)\right\}^2=
(Q-2gv)\,\Big(v'^2 + \mcg_{kh}(v)^2\Big),\label{main}\ee where we
have used the linearity of $\mcg_{kh}$ and the fact that its action
on constant functions is given by (\ref{dnc}).By gathering
(\ref{averv}), (\ref{inj1}), (\ref{mapbd1}), (\ref{posi}) and
(\ref{main}), we obtain (\ref{maineq}).

Conversely, suppose that the positive number $h$ and the function
$v\in C^{1,\alpha}_{2\pi}$ satisfy (\ref{maineq}). We now show how
one can construct a solution of (\ref{apb}) by reversing the process
which led from (\ref{apb}) to (\ref{maineq}). Let $V$ be the
harmonic function in $\mcr_{kh}$ which satisfies (\ref{cou1}) and
(\ref{vv}), and let $U:\mcr_{kh}\to\bdr$ be such that $U+iV$ is
holomorphic. By Lemma \ref{lreg}, $U,\, V\in
C^{1,\alpha}(\overline{\mcr_{kh}})$. Condition (\ref{m1}) ensures
that (\ref{xcou}) holds. Condition (\ref{m2}) shows that the curve
$\mcs$ given by (\ref{S}), which can also be expressed as
(\ref{S1}), is non-self-intersecting, and (\ref{pos}) shows that it
is contained in the upper half-plane. (The fact that $U$ is
determined by $V$ only up to an additive constant leads to the
presence of an arbitrary constant $a\in\bdr$ in (\ref{S1}).) Let
$\Om$ be the domain whose boundary consists of $\mcs$ and $\mcb$. An
application of the Darboux-Picard Theorem \cite[Corollary 9.16, p.\
310]{Bu}, see \cite[Proof of Theorem 3.4]{V2} for details, shows
that $U+iV$ is a conformal mapping from $\mcr_{kh}$ onto $\Om$,
which admits an extension as a homeomorphism between the closures of
these domains, with $\{(x,0):x\in\bdr\}$ being mapped onto $\mcs$
and $\{(x,-kh):x\in\bdr\}$ being mapped onto $\mcb$. It follows from
(\ref{xcou}) that $\Om$ is a $L$-periodic strip-like domain, where
$L$ is given by (\ref{lk}). The domain $\Om$ has conformal mean
depth $h$, since the mapping $\tilde U+i\tilde V:\mcr_h\to\Om$,
where $\tilde U, \tilde V$ are given by (\ref{uv}), has all the
required properties. Note also that (\ref{m3}) shows that $\mcs$ is
a $C^{1,\alpha}$ curve. Let us define $\zeta$ as the unique solution
of (\ref{nap0})-(\ref{nap2}). Obviously, $\zeta\in
C^{1,\alpha}(\overline{\mcr_{kh}})\cap C^\infty(\mcr_{kh})$. Let us
define $\xi$ by (\ref{zeta}), and then $\psi$ by (\ref{xi}). We
obtain that $\psi\in C^{1,\alpha}(\overline\Om)\cap C^\infty(\Om)$
satisfies (\ref{ap0})-(\ref{ap2}). Finally, since (\ref{m0}) holds,
we obtain that $\psi$ satisfies (\ref{ap3}). This completes the
proof.
\end{proof}

\bigskip\bigskip

\begin{center}
\includegraphics[width=13cm]{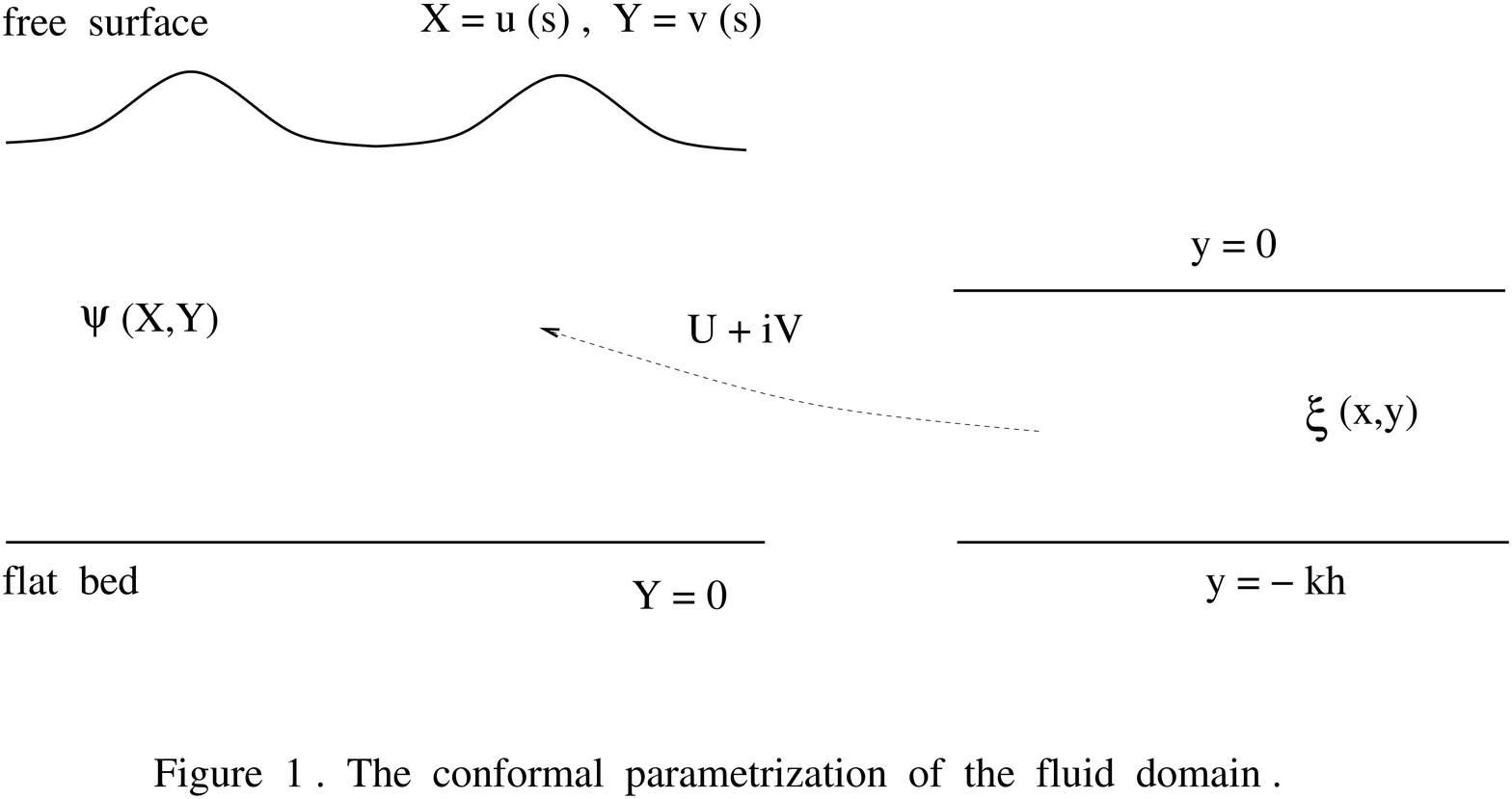}
\end{center}

\bigskip\bigskip

For the sake of completeness, notice that differentiating
(\ref{xi}), solving the resulting linear system for $\psi_X$ and
$\psi_Y$ and taking subsequently (\ref{zeta}) into accoount, we find
that the velocity field $(\psi_Y,\,-\,\psi_X)$ at the location
$(X,Y)=\Big(U(x,y),\,V(x,y)\Big) \in \Omega$, where $(x,y) \in
\mcr_{kh}$, can be expressed as
$$\Big(\frac{V_x\zeta_x+V_y\zeta_y}{V_x^2+V_y^2}\,-\,\gamma V,\, \frac{V_x\zeta_y-V_y\zeta_x}{V_x^2+V_y^2}\Big)$$
in terms of $\zeta(x,y)$ and of the conformal map $U+iV$ from
$\mcr_{kh}$ to $\Omega$.

\section{On periodic harmonic functions in a strip}

This section contains a more detailed investigation of the operators
$\mcc_d$ and $\mcg_d$, for $d>0$, in particular their representation
as singular integrals and some of its consequences. These results
are important in the study, carried out in Sections 4 and 5, of the
regularity and local bifurcation of solutions of (\ref{maineq}).

Let $L^2_{2\pi}$ be the space of $2\pi$-periodic locally square
integrable functions of one real variable, and $W^{1,2}_{2\pi}$ be
the space of $2\pi$-periodic locally absolutely continuous functions
whose weak derivatives are in  $L^2_{2\pi}$. By $L^2_{2\pi,\circ}$
and $W^{1,2}_{2\pi,\circ}$ be denote the subspaces of $L^2_{2\pi}$
and $W^{1,2}_{2\pi}$ whose elements have zero mean over one period.
Every function $w\in L^2_{2\pi}$ has a Fourier series expansion \be
w= [w]+\sum_{n=1}^\infty a_n \cos (nx) + \sum_{n=1}^\infty b_n \sin
(nx),\label{fexp}\ee and, if we denote
\[||w||=\left([w]^2+\sum_{n=1}^\infty(a_n^2+b_n^2)\right)^{1/2}\]
then, by Parseval's Identity, $||\cdot||$ is an equivalent norm to
the standard one in $L^2_{2\pi}$.

For any $d>0$ and $w\in L^2_{2\pi}$, the function $W:\mcr_d\to\bdr$
given by
\[W(x,y)=\frac{[w]}{d}(y+d)+\sum_{n=1}^\infty a_n \frac{\sinh (n(y+d))}{\sinh
(nd)}\cos (nx)
 +\sum_{n=1}^\infty b_n\frac{\sinh (n(y+d))}{\sinh nd}\sin (nx),\]
is the unique solution of (\ref{wdn}) with the third condition there
being satisfied not necessarily in the classical sense, but rather
as \be \lim_{y\nearrow 0}||W(\cdot, y)-w||=0.\ee
 When $w\in W^{1,2}_{2\pi}$, we define $\mcg_d(w)$ as the unique function in $L^2_{2\pi}$ such that
 \be \lim_{y\nearrow 0}||W_y(\cdot, y)-\mcg_d(w)||=0,\ee
namely
\begin{equation}\label{G}
 \mathcal{G}_d(w)= \frac{[w]}{d}+\sum_{n=1}^\infty na_n \coth
 (nd)\cos(nx)
 +\sum_{n=1}^\infty n b_n\coth (nd)\sin (nx).
 \end{equation}
The mapping $\mcg_d$ is the periodic Dirichlet-Neumann operator for
a strip, a bounded linear operator from $W^{1,2}_{2\pi}$ into
$L^2_{2\pi}$.

Suppose that $w\in L^2_{2\pi,\circ}$. Then any harmonic function $Z$
in $\mcr_d$ such that $Z+iW$ is holomorphic is given by
\[Z(x,y)=C+\sum_{n=1}^\infty a_n \frac{\cosh (n(y+d))}{\sinh
(nd)}\sin (nx)
 -\sum_{n=1}^\infty b_n\frac{\cosh (n(y+d))}{\sinh nd}\cos (nx),\]
 where $C$ is a constant. We choose $C=0$ and define $\mcc_d(w)$ as
 the unique function in $L^2_{2\pi,\circ}$ such that
  \be \lim_{y\nearrow 0}||Z(\cdot, y)-\mcc_d(w)||=0,\ee
  namely
  \begin{equation}\label{Cd}
 \mathcal{C}_d(w)=\sum_{n=1}^\infty a_n \coth
 (nd)\sin(nx)
 -\sum_{n=1}^\infty b_n\coth (nd)\cos (nx).
 \end{equation}
 The mapping $\mcc_d$ is the periodic Hilbert transform for a strip, a bounded linear operator from
 $L^2_{2\pi,\circ}$ into itself. Moreover, one can easily see that $\mcc_d$ is a bijection
 from $L^2_{2\pi,\circ}$ onto itself. Let us denote by $\mcc_d^{-1}$
 its inverse. Observe that, by formally setting $d=\infty$ in (\ref{Cd}), we obtain
 the familiar
 periodic Hilbert transform \cite{BT,ST,To},
  \begin{equation}\label{C}
 \mathcal{C}(w)=\sum_{n=1}^\infty a_n \sin(nx)
 -\sum_{n=1}^\infty b_n\cos (nx).
 \end{equation}
for all $w\in L^2_{2\pi,\circ}$ with the Fourier expansion
(\ref{fexp}). The operator $\mcc$ is a bijection
 from $L^2_{2\pi,\circ}$ onto itself, and $\mcc^{-1}=-\mcc$.

The operator $\mcc$ has a pointwise almost everywhere representation
as a singular integral
\begin{equation}\label{ht}
\mcc (w)(t)=\frac{1}{2\pi}\,PV\,\int_{-\pi}^\pi
{\cot\,\left(\frac{t-s}{2}\right)}{w(s)}\,ds,
\end{equation}
where $PV$ denotes a principal value integral \cite{St}. The
representation (\ref{ht}) is useful in proving Privalov's Theorem
\cite{Ko}: for every $\alpha \in (0,1)$ the operator $\mcc$ is a
bounded linear operator from $C^{0,\alpha}_{2\pi}$ into itself.
Moreover, (\ref{ht}) is also instrumental in proving (see
\cite[Lemma 10.5.3, p.\ 148]{BT}) that the operator
\begin{equation}\label{Q0}
w\mapsto{\mathcal Q}( w)=w\,\mcc (w') - \mcc(ww')
\end{equation}
maps functions $w \in C^{1,\alpha}_{2\pi}$ into functions in
$C^{1,\delta}_{2\pi}$ for any $\delta \in (0,\alpha)$. We now prove
that these results, and the corresponding ones for functions of
higher regularity, extend to the case of $\mcc_d$, for any $d>0$, by
exhibiting an analogous representation of $\mcc_d$ as a singular
integral. Let us denote by $\mathcal{Q}_d$ the mapping
\begin{equation}\label{Q}
w \mapsto {\mathcal Q}_d( w)=w\,\mcc_d (w') -\mcc_d(ww').
\end{equation}

\begin{lemma}\label{lpriv} For any $d>0$, $p\ge 0$ integer and $\alpha\in (0,1)$, $\mcc_d$ is a
bounded linear operator from $C^{p,\alpha}_{2\pi,\circ}$ into
itself. Moreover, $\mcc_d^{-1}$ is also a bounded linear operator
from $C^{p,\alpha}_{2\pi,\circ}$ into itself.
\end{lemma}

\begin{lemma}\label{lcom}
  If $w \in C^{p,\alpha}_{2\pi}$ with $p\ge 1$ integer and $\alpha \in
(0,1)$, then ${\mathcal Q}_d(w) \in C^{p,\delta}_{2\pi}$ for any
$\delta \in (0,\alpha)$.
\end{lemma}

Observe that the operators $\mcc$ and $\mcc_d$ can be extended by
(complex-)linearity to complex-valued functions in
$L^2_{2\pi,\circ}$, and are characterized by their action on the
trigonometric system $\{e^{int}\}_{n\in\mathbb{Z}\setminus\{0\}}$:
\begin{equation}\label{c1}
\mcc(e^{int})=-{i}\,\textnormal{sgn} (n)\,e^{int},\qquad n \in \bdz
\setminus \{0\},
\end{equation}
and
\begin{equation}\label{c}
 \mcc_d(e^{int})=-{i}\,{\coth(nd)}\,e^{int},\qquad n \in \bdz
\setminus \{0\}.
\end{equation}
Let us write \be\mcc_d=\mcc+\mathcal{K}_d.\label{cade}\ee Then the
operator $\mathcal{K}_d$ corresponds to the Fourier multiplier
operator on $L^2_{2\pi,\circ}$: \be \Big\{w = \sum_{n \in
\bdz\setminus\{0\}} c_n\,e^{int}\Big\} \mapsto \Big\{\sum_{n \in
\bdz\setminus \{0\}}
-i\,\textnormal{sgn}(n)\,\lambda_n\,c_n\,e^{int}\Big\}\label{kap}\ee
with
$$\lambda_n=\frac{2}{e^{2|n|d}-1},\qquad |n| \ge 1.$$
 Notice that for any $n \ge 1$,
$$e^{2nd}-1=\sum_{m=0}^\infty \frac{1}{m!}\,n^m\,(2d)^m \ge
\frac{1}{(p+1)!}\,n^{p+1}\,(2d)^{p+1},\quad p \ge 1,$$ yields
$$0 \le \frac{n^p}{e^{2nd}-1} \le \frac{1}{n}\,\frac{(p+1)!}{(2d)^{p+1}},\quad p \ge 1,$$
so that
$$\sum_{n\in\bdz\setminus\{0\}} |n|^{2p}\lambda_n^2 \le 4\,\Big(\frac{(p+1)!}{(2d)^{p+1}}
\Big)^2\,\sum_{n\in\bdz\setminus\{0\}} \frac{1}{n^2} < \infty.$$ Let
$\kappa_d\in L^2_{2\pi,\circ}$ be given by
\be\kappa_d(t)=\sum_{n\in\bdz\setminus\{0\}}-i\,\textnormal{sgn}(n)\,\lambda_n
\,e^{int} =\sum_{n=1}^{\infty}2\lambda_n\sin(nt),
\quad{t\in\bdr}.\label{kapad}\ee Since the Fourier coefficients of
$\kappa_d$ decays faster than any power of $|n|$, $\kappa_d$ is a
function of class $C^\infty$ (see \cite{DM}). It follows from
(\ref{kap}) that $\mck_d(w)$ is the periodic convolution of $w$ with
the smooth function $\kappa_d$ given by (\ref{kapad}), namely
\begin{equation}\label{c2}
\mck_d (w)(t)=\frac{1}{2\pi}\int_{-\pi}^\pi \kappa_d
(t-s)\,w(s)\,ds,\qquad t\in\bdr.
\end{equation}

\begin{proof}[Proof of Lemma \ref{lpriv}.] It is easy to see that
$\mck_d$ is a bounded linear operator from
$C^{p,\alpha}_{2\pi,\circ}$ into itself. The operator $\mcc$ is also
a bounded linear operator from $C^{p,\alpha}_{2\pi,\circ}$ into
itself, by Privalov's Theorem
 and the fact that $\mcc$ commutes with differentiation on
$C^{1,\alpha}_{2\pi,\circ}$. The required result for $\mcc_d$
follows from (\ref{cade}). The corresponding result for
$\mcc_d^{-1}$ is obtained by the same argument, after observing that
one can write \be \mcc_d^{-1}=-\mcc+\tilde \mck_d,\ee where
$\tilde\mck_d$ is given by periodic convolution with a smooth
function.

\end{proof}

\begin{proof}[Proof of Lemma \ref{lcom}.] Using (\ref{Q}) and (\ref{c2}), we
write \be {\mathcal Q}_d(w) = {\mathcal Q}(w)+ \{w\,\mck_d
(w')-\mck_d(ww')\}.\label{star}\ee Since $\kappa_d$ is smooth, it is
easy to see from (\ref{c2}) that $\mck_d(w')$ and $\mck_d(ww')$ are
both smooth functions. The required result follows then from
(\ref{star}), provided that one can show that $\mathcal{Q}$ maps
$C^{p,\delta}_{2\pi}$ into $C^{p,\delta}_{2\pi}$ for any $\delta \in
(0,\alpha)$. This fact has been proved in \cite{BT} for $p=1$, and
in \cite{PT} for $p\geq 2$, the proof in \cite{PT} being in fact in
a more general setting.

 For the sake of completeness, we now give a self-contained proof of the fact that, for any $p\geq 1$
 integer, $\mathcal{Q}$ maps $C^{p,\delta}_{2\pi}$ into $C^{p,\delta}_{2\pi}$
for any $\delta \in (0,\alpha)$. Our proof relies on a commutator
estimate proved in Appendix B, which is more general and has a more
transparent proof than the results in \cite{PT}. Let $w\in
C^{p,\delta}_{2\pi}$, with $p\geq 1$ integer and $\alpha\in (0,1)$.
 Since $w^{(p)} \in C^{0,\alpha}_{2\pi}$, we may
interchange the operators $\mcc$ and differentiation to infer from
Leibniz's rule that
$$\Big(w\,\mcc(w^{(p)})-\mcc(ww^{(p)})\Big) -\,\pa_x^{p-1}\Big( w\,\mcc(w')-\mcc(ww') \Big)  \in C^{1,\alpha}_{2\pi}.$$
Thus it suffices to show that for any $\delta \in (0,\alpha)$ we
have \begin{equation}\label{hk}
\Big(w\,\mcc(w^{(p)})-\mcc(ww^{(p)})\Big)\in C^{1,\delta}_{2\pi}.
\end{equation}
However, estimate (\ref{hk}) is an immediate consequence of Lemma
B.1 in Appendix B. This completes the proof.
\end{proof}

\begin{remark}{\rm Since for $\alpha \in (0,1)$ the function
$$w: [-\pi,\pi] \to \bdr,\qquad w(x)=\frac{x}{\alpha+1}\,\Big(|x|^\alpha - \pi^\alpha\Big),$$
extends by periodicity to a function $w \in C^{1,\alpha}_{2\pi}$
with $w\,\mcc(w')-\mcc(ww') \not \in  C^{1,\alpha}_{2\pi}$ (see the
detailed estimates in \cite{PT}), in view of (\ref{c2}) we see that
the result of Lemma 3.2 cannot be improved to accommodate
$\delta=\alpha$.}\end{remark}

\section{Regularity}

The main result of this section concerns regularity of solutions of
(\ref{maineq}).
\begin{theorem}\label{reg} Let $h>0$ and $v \in C^{1,\alpha}_{2\pi}$ be a solution of {\rm
(\ref{maineq})} such that
\begin{equation}\label{sm}
 Q-2g\,v(x)\,>0 \qquad\text{for all }x\in\bdr.
\end{equation}
Then $v\in C^\infty_{2\pi}$.
\end{theorem}

\begin{remark}{\rm Theorem \ref{reg} implies that, is $(\Om,\psi)$ is a solution
of (\ref{apb}) of class $C^{1,\alpha}$ with no stagnation points on
the free surface, then $\mcs$ is a $C^\infty$ curve and $\psi\in
C^\infty(\overline\Om)$. Indeed, by Theorem \ref{refor}, any such
solution gives rise to a solution of (\ref{maineq}), the absence of
stagnation points on $\mcs$ being equivalent to (\ref{sm}). By
Theorem \ref{reg}, $v \in C^{\infty}_{2\pi}$ and, by Lemma
\ref{lpriv}, $\mcc_{kh}(v-h)\in C^{\infty}_{2\pi}$. Hence $\mcs$,
being given by (\ref{paras}), is a $C^\infty$ curve, and the fact
that $\psi\in C^\infty(\overline\Om)$
 follows from standard elliptic regularity theory \cite{GT}.}
 \end{remark}

\begin{remark}{\rm Our proof of Theorem \ref{reg} is based on harmonic
analysis estimates and is relatively simple. Using different methods
which are somewhat more involved, the stronger conclusion that
$s\mapsto v(s)$ is real-analytic on $\bdr$ can be obtained. Indeed,
 once the $C^{2,\alpha}$ regularity of $\mcs$ and of $\psi$ in
 $\overline\Om$ is known (for example from Theorem
 4.1), the approach of Kinderlehrer, Nirenberg and Spruck \cite{KNS} is
applicable, as first observed by Constantin and Escher \cite{CE1},
showing that the curve $\mcs$ is real-analytic. Then a standard
result of regularity of conformal mappings up to the boundary
\cite{Po} shows that $U+iV$ has a holomorphic extension to a
neighbourhood of the real axis, implying in particular that
$s\mapsto v(s)$ is real-analytic on $\bdr$. An alternative proof of
this fact, based on ordinary differential equations in the complex
domain in the spirit of Lewy's theorem \cite{Le}, has been pointed
out to us by the anonymous referee of this paper.}
\end{remark}

\begin{proof}[Proof of Theorem \ref{reg}.]
Let $U,\, V$ be as in the proof of the second part of Theorem
\ref{refor}, with  $U,\, V\in C^{1,\alpha}(\overline{\mcr_{kh}})$.
As noted there, $U+iV$ is a conformal mapping from $\mcr_{kh}$ onto
$\Om$, and therefore \[ U_x+iV_x\neq 0\qquad\text{in
}\mcr_{kh}.\]Since $V>0$ in $\mcr_{kh}$ and $V(x,-kh)=0$ for all
$x\in\bdr$, it follows from Hopf boundary-point lemma \cite{GT} that
\[ V_y(x,-kh)>0\qquad\text{for all }x\in\bdr.\]Taking also into
account (\ref{m3}), we have thus obtained that \be U_x+iV_x\neq
0\qquad\text{in }\overline{\mcr_{kh}}.\label{noz}\ee One can
therefore write
$$U_x+iV_x=(U_x^2+V_x^2)^{1/2}\,e^{i\theta},$$
where $\theta \in C^{0,\alpha}(\overline{\mcr_h})$ is such that
\[\log\,(U_x^2+V_x^2)^{1/2}+i\theta\quad\text{ is holomorphic in $\mcr_{kh}$},\]
\[\theta(x,-kh)=0\quad\text{for all }x\in\bdr.\]
Let \be\label{reg1} \theta_0(x)=\theta(x,0),\qquad
u(x)=U(x,0),\qquad x\in\bdr.\ee Then \be
\mcc_{kh}(\theta_0)=\log\,(u'^2+v'^2)^{1/2}-[\log\,(u'^2+v'^2)^{1/2}],\ee
so that
\be\theta_0=\mcc_{kh}^{-1}\left(\log\,(u'^2+v'^2)^{1/2}-[\log\,(u'^2+v'^2)^{1/2}]\right).\label{conj}\ee
Since (\ref{sm}) holds, equation (\ref{m0}) can be written as
\begin{equation}\label{reg2}
(u'^2+v'^2)^{1/2}=\displaystyle\frac{\Big|\,\displaystyle\frac{m}{kh}
+\gamma\{\mathcal{G}_{kh}(v^2/2)-v\,\mathcal{G}_{kh}(v)\}\Big|}{(Q-2gv)^{1/2}}.
\end{equation}
Notice that in view of (\ref{cone}), (\ref{m1}) and (\ref{Q}), we
have \be
\mathcal{G}_{kh}(v^2/2)-v\,\mathcal{G}_{kh}(v)=\frac{[v^2]}{2kh}+\mcc_{kh}(vv')-\frac{v}{k}-v\,\mcc_{kh}(v')=
\frac{[v^2]}{2kh}-\frac{v}{k}-\mathcal{Q}_{kh}(v). \label{reg3}\ee
Since $v\in C^{1,\alpha}_{2\pi}$, Lemma \ref{lcom} ensures that
$\mathcal{Q}_{kh}(v) \in C^{1,\,\alpha/2}_{2\pi}$ and from
(\ref{reg2}) and (\ref{reg3}) we infer that $(u'^2+v'^2) \in
C^{1,\alpha/2}_{2\pi}$. Using this in (\ref{conj}), Lemma
\ref{lpriv} ensures now that $\theta_0 \in C^{1,\alpha/2}_{2\pi}$.
But
$$\begin{cases}
u' = (u'^2+v'^2)^{1/2}\,\cos\,\theta_0,\\
v'= (u'^2+v'^2)^{1/2}\,\sin\,\theta_0,
\end{cases}$$
so that $u,\,v \in C^{2,\,\alpha/2}_{2\pi}$. Implementing this
procedure $p$ times we obtain that
$$u,\,v \in C^{1+p,\,\alpha/2^p}_{2\pi},\qquad p \ge 1.$$
Consequently $v\in C^\infty_{2\pi}$. This completes the
proof.\end{proof}

\section{Local bifurcation}

In this section we prove the existence of solutions of
(\ref{maineq}). Because of (\ref{m1}), it is natural to put \be v=
w+h,\ee and rewrite (\ref{maineq}), taking into account
(\ref{cone}), as
\begin{subequations}\label{naineq}
\begin{align}&\left\{\frac{m}{kh} +\gamma \left(
\frac{[w^2]}{2kh}-\frac{w}{k}-\frac{h}{2k} +\mcc_{kh}(ww')-w\mcc_{kh}(w')\right)\right\}^2=\label{n0}\\
&\qquad\qquad=(Q-2gh-2gw)\,\left\{w'^2 +
\left(\frac{1}{k}+\mcc_{kh}(w')\right)^2\right\},\non\\&
[w]=0,\label{n1}\\
&w(x)>-h\quad\text{for all }x\in\bdr,\label{os}\\ &\text{the mapping
$x\mapsto \left(\frac{x}{k}+\mcc_{kh}(w)(x), w(x)+h\right)$ is
injective on $\bdr$},\label{n2}\\&
w'(x)^2+\left(\frac{1}{k}+\mcc_{kh}(w')(x)\right)^2\neq
0\qquad\text{for all }x\in\bdr,\label{n3}
\end{align}
\end{subequations}
We prove, for each $\gamma\in\bdr$, $k>0$ and $h>0$ fixed, and
regarding $m$ and $Q$ as parameters, the existence of solutions
$w\in C^{1,\alpha}_{2\pi}$ of (\ref{naineq}). Note that, if \be
\frac{1}{k}+\mcc_{kh}(w')(x)>0\qquad\text{ for all
}x\in\bdr,\label{gra}\ee then (\ref{n2}) and (\ref{n3}) are
automatically satisfied and, moreover, the corresponding free
boundary $\mcs$ is the graph of a function.
  Motivated by the remark that
 $w$ sufficiently small in $C^{1,\alpha}_{2\pi}$ ensures the
 validity of
(\ref{os}) and (\ref{gra}), we concentrate on proving the existence
of small-amplitude solutions $w\in C^{1,\alpha}_{2\pi,\circ}$ of
(\ref{n0}). Of course, any study of large-amplitude solutions of
(\ref{n0})-(\ref{n1}) will need to discuss also the validity of
(\ref{os})-(\ref{n3}), which is necessary for the construction of
solutions of (\ref{apb}).

Note that $w=0\in C^{1,\alpha}_{2\pi,\circ}$ is a solution of
(\ref{n0}) if and only if \be Q=2gh + \Big( \frac{m}{h}-\frac{\gamma
h}{2}\Big)^2. \label{zar}\ee This suggests setting
\begin{subequations}
\begin{align}\label{lb4}
\lambda&=\frac{m}{h}-\frac{\gamma h}{2},\\
 \mu&= Q-2gh -\Big( \frac{m}{h}-\frac{\gamma h}{2}\Big)^2.\label{lb1}\end{align}
 \end{subequations}
 The mapping $(m, Q)\mapsto(\lambda,\mu)$ is a bijection
 from $\bdr^2$ onto itself. Equations (\ref{n0})-(\ref{n1}) can be rewritten as
 \begin{align}&\left\{\frac{\lambda}{k} +\gamma \left(
\frac{[w^2]}{2kh}-\frac{w}{k} +\mcc_{kh}(ww')-w\mcc_{kh}(w')\right)\right\}^2=\label{newe}\\
&\qquad\qquad=(\lambda^2+\mu-2gw)\,\left\{w'^2 +
\left(\frac{1}{k}+\mcc_{kh}(w')\right)^2\right\},\non\end{align}
with $w\in C^{1,\alpha}_{2\pi,\circ}$, $\mu\in\bdr$ and
$\la\in\bdr$. Note that $w=0\in C^{1,\alpha}_{2\pi,\circ}$ and
$\mu=0$ is a solution of (\ref{newe}) for every $\la\in\bdr$.

To prove the existence of solutions of (\ref{newe}) we apply the
Crandall-Rabinowitz theorem \cite{CR} on bifurcation from a simple
eigenvalue. For a linear operator $\mcl$ between two Banach spaces,
let us denote by ${\mathcal N}(\mcl)$ its null space and by
${\mathcal R}(\mcl)$ its range.\medskip

\begin{theorem}[{\bf Local bifurcation theorem} \cite{CR}] Let $\mcx$ and $\mcy$
be Banach spaces, $I$ an open interval in ${\mathbb R}$ containing
$\lambda^\ast$, and $F \in C(I \times \mcx,\mcy)$. Suppose that

(i) $F(\lambda,0) = 0$ for all $\lambda \in I$;

(ii) $\pa_\lambda F$, $\pa_u F$, and $\pa^2_{\lambda,u} F$ exist and
are continuous;

(iii) ${\mathcal N}(\pa_u F(\lambda^\ast,0))$ and $\mcy/{\mathcal
R}(\pa_u F(\lambda^\ast,0))$ are one-dimensional, with the nullspace
generated by $u^\ast$;

(iv) the transversality condition $\pa^2_{\lambda,u}
F(\lambda^\ast,0)\,(1,u^\ast) \not \in {\mathcal R}(\pa_u
F(\lambda^\ast,0))$ holds.

Then there exists a continuous local bifurcation curve
$\{(\lambda(s),u(s)) :\  |s| < \varepsilon\}$ with $\varepsilon >$ 0
sufficiently small such that $(\lambda(0),u(0)) = (\lambda^\ast, 0)$
and
$$\{(\lambda,u) \in {\mathcal O}:\ u \neq 0,\ F(\lambda,u)=0\}= \{(\lambda(s),u(s)) :\ 0< |s| < \varepsilon \}$$
for some neighborhood ${\mathcal O}$ of $(\lambda^\ast, 0) \in I
\times \mcx$. Moreover, we have
$$u(s) = su^\ast + o(s)\quad\hbox{in}\quad \mcx,\ |s| < \varepsilon,$$
and if $\pa^2_{u}F$ is also continuous, then the curve is of class
$C^1$, while for $F$ of class $C^k$ ($k \ge 2$) or real-analytic, $s
\mapsto u(s)$ is of class $C^{k-1}$, respectively real-analytic.
\end{theorem}

Let us denote, for every integer $p\ge 0$,
\begin{align}C^{p,\alpha}_{2\pi, e}=\{ f\in C^{p,\alpha}_{2\pi}:
f(x)=f(-x)\text{ for all }x\in\bdr\},\\
C^{p,\alpha}_{2\pi,\circ, e}=\{ f\in C^{p,\alpha}_{2\pi,\circ}:
f(x)=f(-x)\text{ for all }x\in\bdr\}.
\end{align}
To apply the local bifurcation theorem to (\ref{newe}), let \be
\mcx=\bdr\times C^{p+1,\alpha}_{2\pi,\circ, e},\qquad
\mcy=C^{p,\alpha}_{2\pi, e},\label{bsp}\ee for some $p\ge 0$. Then
we can write (\ref{newe}) as $F(\lambda,(\mu, w))=0$ with $F: \bdr
\times \mcx \to \mcy$ real-analytic given by \begin{align}
F(\lambda,(\mu, w))
\,=&\gamma^2\,\left(\,\mcc_{kh}(ww')-w\,\mathcal{C}_{kh}(w')-\frac{w}{k}+\frac{[w^2]}{2kh}\right)^2
\\&+\frac{2\lambda\gamma}{k}\,\left(\,\mcc_{kh}(ww')-w\,\mathcal{C}_{kh}(w')-\frac{w}{k}+\frac{[w^2]}{2kh}\right)
\non\\
&+\, (2gw-\mu)\, \left(\frac{1}{k^2}+
w'^2+\frac{2}{k}\,\mathcal{C}_{kh}(w')+\Big(\mcc_{kh}(w')\Big)^2\right)\non\\
&- \lambda^2\, \left(
w'^2+\frac{2}{k}\,\mathcal{C}_{kh}(w')+\Big(\mcc_{kh}(w')\Big)^2\right).\non
\end{align}
As noted earlier, $F(\lambda,(0,0))=0$ for all $\lambda\in\bdr$.
Taking into account Lemma \ref{lpriv}, we easily compute
\begin{equation}\label{lb5}
\pa_{(\mu,w)}\,F(\lambda,(0,0))\,(f,\nu)
=\frac{2}{k^2}\Big((g-\lambda\gamma)f-\lambda^2k\mcc_{kh}(f')\Big)-\frac{\nu}{k^2},\qquad
(f,\nu) \in \mcx.
\end{equation}

It follows from the representation (\ref{Cd}) and Lemma \ref{lpriv}
that the bounded linear operator
$\pa_{(\mu,w)}\,F(\lambda,(0,0)):\mcx\to \mcy$ is invertible
whenever $\lambda$ does \emph{not} satisfy \be\lambda^2
nk\coth(nkh)= g-\lambda\gamma,\label{bif}\ee for any integer $n\ge
1$. Hence all potential bifurcation points for (\ref{newe}) are to
be found among the solutions of (\ref{bif}) for some integer $n\ge
1$.

Suppose now that $\la^*$ is a solution of (\ref{bif}) for some
integer $n\ge 1$. (Observe that (\ref{bif}) has exactly two
different solutions for each $n\ge 1$.) Using again the
representation (\ref{Cd}) and Lemma \ref{lpriv}, it follows that $
{\mathcal N}(\pa_{(\mu,w)} F(\lambda^\ast,(0,0)))$ is
one-dimensional and generated by $(0, w^*)\in \mcx$, where
$w^*(x)=\cos (nx)$ for all $x\in\bdr$, while ${\mathcal
R}(\pa_{(\mu,w)} F(\lambda^\ast,(0,0)))$ is the closed subspace of
$\mcy$ formed by the functions $f \in \mcy$ satisfying
$$\int_{-\pi}^{\pi} f(x)\,\cos(nx)\,dx=0,$$
so that $\mcy/{\mathcal R}(\pa_{(\mu,w)} F(\lambda^\ast,(0,0)))$ is
the one-dimensional subspace of $\mcy$ generated by the function
$w^\ast(x)= \cos(nx)$. Using (\ref{lb5}), we now compute
$$\pa^2_{\lambda,(\mu,w)}\,F(\lambda^\ast,(0,0)) \,(1,(0,w^\ast))=\frac{2}{k^2}(-\gamma -
2\lambda^\ast nk\coth(nkh))\,w^\ast \not \in {\mathcal
R}(\pa_{(\mu,w)} F(\lambda^\ast,(0,0)))$$ since, using (\ref{bif}),
we have
$$-\gamma - 2\lambda^\ast nk\coth(nkh)=-\,\lambda^\ast\,\Big(nk\coth(nkh)+\frac{g}{(\lambda^\ast)^2}\Big)
\neq 0.$$ Therefore, by the local bifurcation theorem, the solutions
of (\ref{bif}) for any integer $n\ge 1$ are bifurcation points.
Motivated by the quest for solutions of (\ref{newe}) of
\emph{minimal} period $2\pi$, we take $n=1$ in (\ref{bif}),
obtaining the bifurcation values
\begin{equation}
\lambda_{\pm}=-\,\frac{\gamma\tanh(kh)}{2k} \pm
\displaystyle\sqrt{\frac{
\gamma^2\tanh^2(kh)}{4k^2}+g\,\frac{\tanh(kh)}{k}}\label{lb6}
\end{equation}
The corresponding values of $m$ are obtained from (\ref{lb4}),
\begin{equation}\label{lb7}
m_{\pm} =\frac{\gamma h^2}{2}-\,\frac{\gamma h\tanh(kh)}{2k}\pm
h\,\displaystyle\sqrt{\frac{
\gamma^2\tanh^2(kh)}{4k^2}+g\,\frac{\tanh(kh)}{k}}
\end{equation}

The existence of water waves of small amplitude is now immediate.

\begin{theorem}\label{bifu} Given $h>0$, $k>0$, $\gamma \in \bdr$
and $m \in \bdr$ there exist laminar flows\footnote{In the sense
that the water flows in layers parallel to the flat bed (with zero
vertical fluid velocity).} with a flat free surface in water of
depth $h$, of constant vorticity $\gamma$ and relative mass flux
$m$. Moreover, the values $m_\pm$ of the flux given by {\rm
(\ref{lb7})} trigger the appearance of periodic steady waves of
small amplitude, with period ${2\pi}/k$ and conformal mean depth
$h$, which have a smooth profile with one crest and one trough per
period, monotone between consecutive crests and troughs and
symmetric about any crest line. The laminar flows of flux $m_\pm$
are exactly those with horizontal speeds at the flat free surface
$\lambda_\pm$ given by {\rm (\ref{lb6})}.

\end{theorem}

\begin{proof}[Proof of Theorem \ref{bifu}.] The function $w=0$ satisfies (\ref{naineq}) for any value of $m\in\bdr$,
provided that $Q$ is given by (\ref{zar}). These solutions
correspond to laminar flows in the fluid domain bounded below by the
rigid bed $\mcb$ and above by the free surface $Y=h$, with stream
function
\[\psi(X,Y)=-\frac{\gamma}{2}Y^2+\left(\frac{m}{h}+\frac{\gamma
h}{2}\right)Y-m, \qquad X\in\bdr, 0\leq Y\leq h,\] and velocity
field
\begin{equation}\label{lb8}
(\psi_Y,-\psi_X)=\Big(- \gamma Y+\frac{m}{h}+\frac{\gamma h}{2}
,0\Big), \qquad X\in\bdr, 0\leq Y\leq h.
\end{equation}
Observe that for these flows the horizontal velocity at the free
surface is indeed $\la$ given by (\ref{lb4}).

Consider now the local bifurcation curve
$$\{(\lambda(s), (0+o(s), s\cos(x)+o(s))):\ |s| < \varepsilon\} \subset\bdr\times \mcx$$
of solutions of (\ref{newe}), issuing from the points
$(\lambda_{\pm},(0,0))$, with $\lambda_\pm$ given by (\ref{lb6}). If
$\varepsilon>0$ is chosen small enough, then conditions (\ref{os})
and (\ref{gra}) are satisfied, where we have used Lemma \ref{lpriv}.
Therefore the corresponding non-flat free surface $\mcs$ given by
(\ref{paras}) with
$$v=w+h,$$
is the graph of a smooth function, symmetric with respect to the
points corresponding to $x=n\pi$, $n\in\bdz$. If we choose $p$ in
(\ref{bsp}) such that $p\geq 1$, then, since \be
w(x;s)=s\cos(x)+o(s)\qquad\text{in }C^{p+1,\alpha}_{2\pi},\ee we can
ensure that
\[ s w'(x;s)<0 \quad\text{for all }x\in (0,\pi), 0<|s|<\veps, \]
by choosing $\varepsilon>0$ suitably small. Together with the
evenness of $x\mapsto w(x;s)$, this proves the $\mcs$ has one crest
and one trough per minimal period and is monotone between
consecutive crests and troughs.

\end{proof}

Let us consider the family of laminar flows from which bifurcation
of small-amplitude periodic waves was shown to occur. It follows
from (\ref{lb8}) that they have the form \be
(\psi_Y,-\psi_X)=\Big(\lambda_\pm+\gamma(h-Y)
 ,0\Big), \qquad X\in\bdr,\  0\leq
Y\leq h.\ee where $\la_{\pm}$ is given by (\ref{lb6}). The formula
(\ref{lb6}), giving the speed $\la_\pm$ at the free surface in terms
of the depth $h$, period $2\pi/k$ and vorticity $\gamma$, is called
the \emph{dispersion relation}. As far as we are aware, its first
occurrences in the literature are in \cite{Th} and \cite{Bi}, where
it was obtained by formal arguments.

A remarkable feature of some of these flows is that they contain
stagnation points.
 Since the flows are laminar, stagnation
points, if present, form horizontal lines. An elementary analysis,
similar to that in \cite{W}, shows that small-amplitude waves
bifurcating from these flows have a critical layer. The streamline
pattern is considerably different in the case when stagnation points
are present compared to the case when they are not. A typical case
of a flow without stagnation points is depicted in Figure 2 (see the
discussion in \cite{CS}).

\bigskip\bigskip

\begin{center}
\includegraphics[width=9cm]{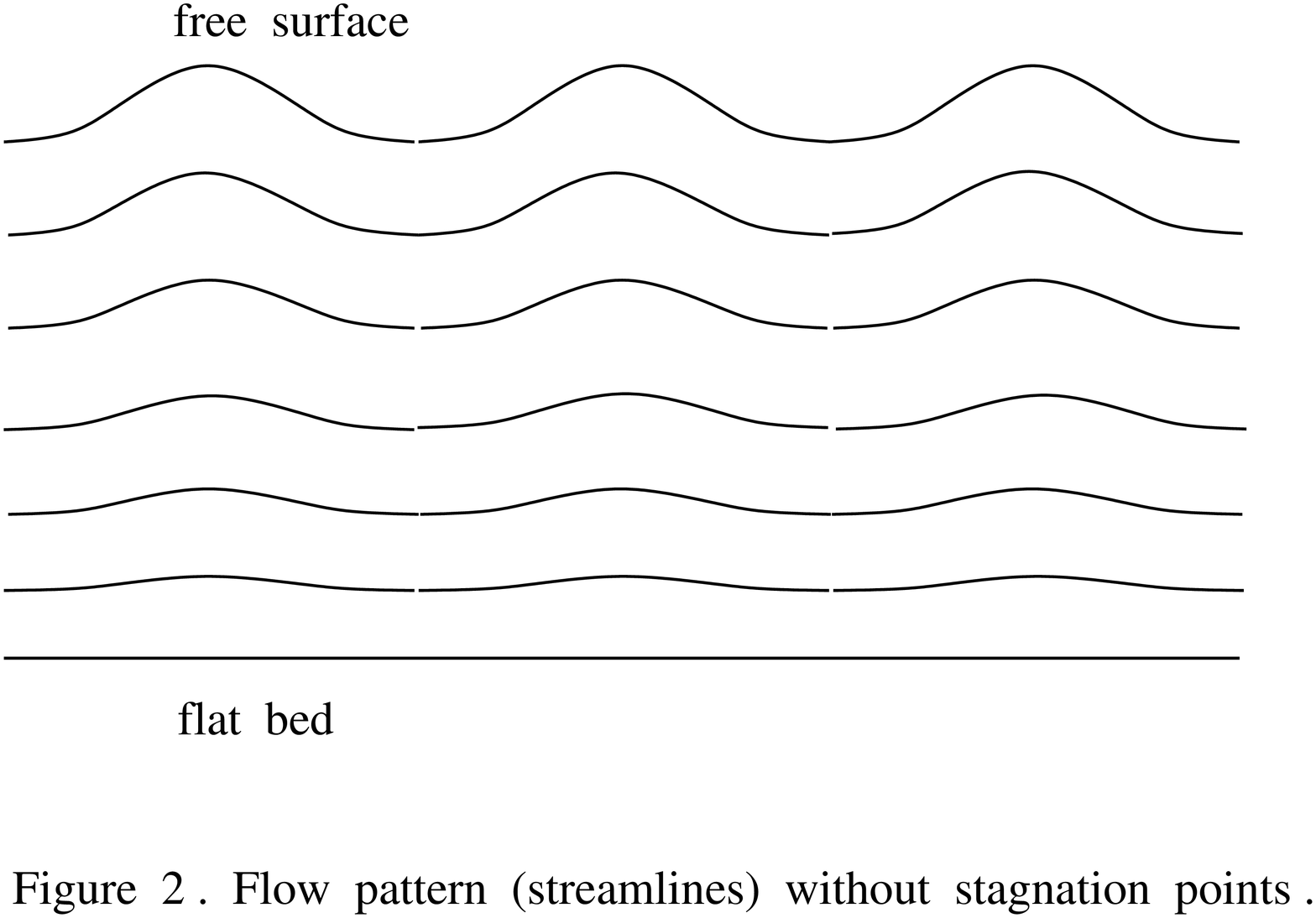}
\end{center}

\bigskip\bigskip

\noindent A picture of the streamline pattern if the flow presents
stagnation points is provided in Figure 3: the `cat's eye' flow
pattern of Thomson (Lord Kelvin) \cite{Ke} is highly typical for
this situation (see the discussion in \cite{W}).

\bigskip\bigskip

\begin{center}
\includegraphics[width=9cm]{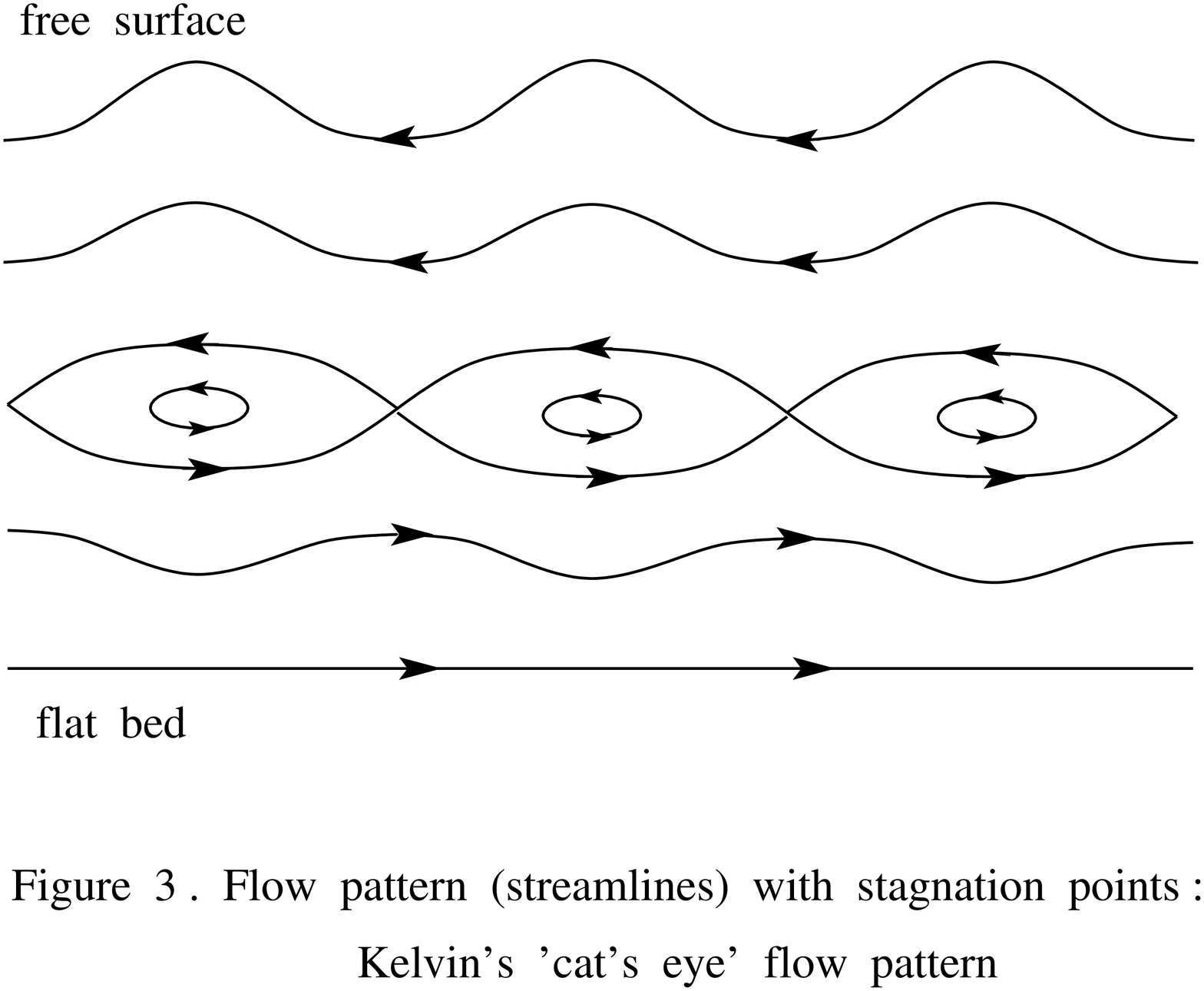}
\end{center}

\bigskip\bigskip

Notice that $\la_+>0$ and $\la_-<0$ irrespective of the values of
$h$, $k$ and $\gamma$, so stagnation points can never occur at the
free surface of the bifurcation-inducing laminar flows. This implies
that in the irrotational case ($\gamma=0$) stagnation points do not
occur anywhere in these flows. In the general case stagnation points
occur if and only if the equation
$$\lambda_\pm+\gamma(h-Y)=0$$
has a solution $Y$ in $[0,h]$. It is easy to see that this is the
case if and only if \be \label{wta}\la_\pm(\la_\pm+\gamma h)\leq
0.\ee

Suppose first that $\gamma >0$. Then $\la_++\gamma h>0$, so the flow
corresponding to $\la_+$ never contains stagnation points. The flow
corresponding to $\la_-$ contains stagnation points if and only if
$\la_-+\gamma h\geq 0$, which is equivalent to \be\label{stp}
\frac{\tanh(kh)}{kh}\leq \frac{\gamma^2h}{g+\gamma^2h}.\ee For fixed
$h$, the left-hand side of (\ref{stp}) is monotone as a function of
$k$, with limit $1$ as $k\to0$ and limit $0$ as $k\to\infty$, while
the right-hand side is monotone as a function of $\gamma$, with
limit $0$ as $\gamma\to 0$ and limit $1$ as $\gamma\to\infty$.
Hence, for fixed $h$ and $\gamma$, there exists a unique $k^*>0$ for
which equality holds in (\ref{stp}), which corresponds to a laminar
flow stagnant at the bottom, and for $k<k^*$ the flow does not have
stagnation points, while for $k>k^*$ the flow has a line of
stagnation points. On the other hand, if $h$ and $k$ are fixed, then
there exists a unique $\gamma^*>0$ for which equality holds in
(\ref{stp}), which corresponds to a laminar flow stagnant at the
bottom, and for $\gamma<\gamma^*$ the flow does not have stagnation
points, while for $\gamma>\gamma^*$ the flow has a line of
stagnation points. Whenever (\ref{stp}) holds, the stagnation line
is $Y=Y_0$, where
\begin{equation}\label{lb12}
h-Y_0=\frac{\tanh(kh)}{2k}+
\displaystyle\sqrt{\frac{\tanh^2(kh)}{4k^2}+\frac{g}{\gamma^2}\,\frac{\tanh(kh)}{k}}
\end{equation}
from which we see that the distance between the stagnation line and
the free surface decreases with $k$ and with $\gamma^2$, and
approaches $0$ if and only if $k\to \infty$ and $\gamma^2\to\infty$.

Suppose now that $\gamma<0$. This case is the mirror image of the
previous one. The flow corresponding to $\la_-$ never contains
stagnation points. The flow corresponding to $\la_+$ contains
stagnation points if and only if (\ref{stp}) holds, in which case
the stagnation line is $Y=Y_0$, where $Y_0$ satisfies (\ref{lb12}).

It is of interest to compare Theorem \ref{bifu} with the local
bifurcation results for waves of constant vorticity $\gamma$ that
were obtained in \cite{CS} and in \cite{W}. Let us fix $k>0$.
Varying $h>0$ in Theorem \ref{bifu} we obtain all periodic traveling
wave solutions of small amplitude and minimal period $2\pi/k$ that
lie in a neighborhood of a laminar flow: for each $h$ we obtain in
the space $(h,m, (\mu, w)) \in  (0,\infty)\times\bdr \times \mcx$
two smooth curves through the points $(h, m_-(h), 0)$ and $(h,
m_+(h),0)$ with $m_{\mp}$ given by (\ref{lb7}). The union of these
curves gives a surface ${\mathcal M} \subset (0,\infty)\times\bdr
\times \mcx$.

\bigskip
\begin{center}
\includegraphics[width=13cm]{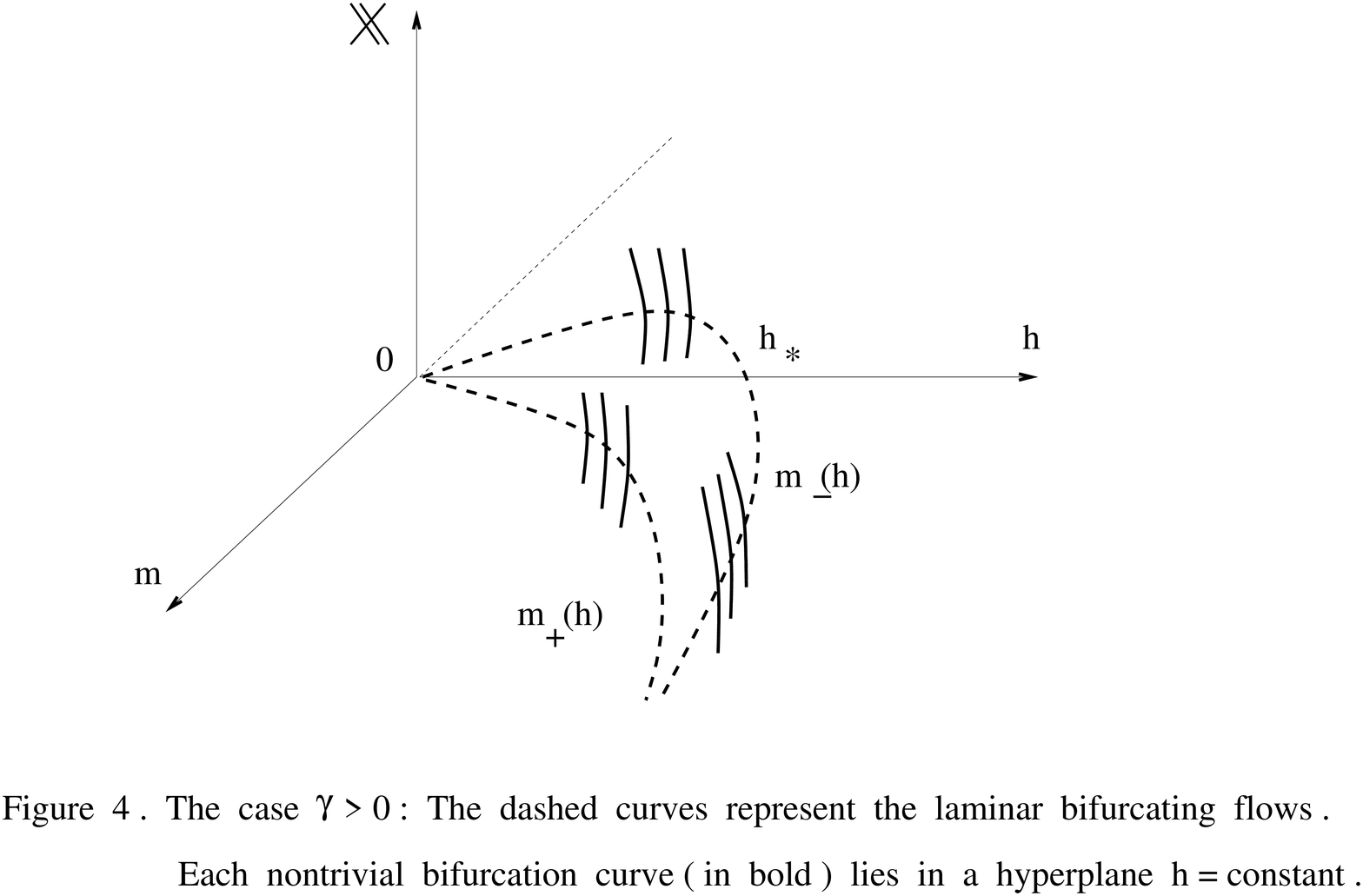}
\end{center}

\bigskip

\noindent In Figure 4 we depict the situation in the case
$\gamma>0$:

(i) In the $(h,m)$-plane the curve $h \mapsto m_+(h)$ starts at the
origin and approaches the parabola $m=\displaystyle\frac{\gamma
h^2}{2}$ as $h \to \infty$, while $m_+(h)>0$ for $h>0$. As pointed
out in the lines preceding (\ref{stp}), the nontrivial flows that
bifurcate from the laminar flows corresponding to $m_+(h)$ do not
present stagnation points.

(ii) In the $(h,m)$-plane the curve $h \mapsto m_-(h)$ also starts
at the origin and approaches the parabola
$m=\displaystyle\frac{\gamma h^2}{2}$ as $h \to \infty$. Denoting by
$h_*>0$ the unique positive solution of the equation
$$\frac{\tanh(kh)}{kh}=\frac{\gamma^2\,h}{4g+2\gamma^2\,h}$$
we see from (\ref{lb7}) that $m_-(h)<0$ for $h \in (0,h_*)$ and
$m_-(h)>0$ for $h>h_*$. Moreover, since
$$\frac{\gamma^2\,h_*}{4g+2\gamma^2\,h_*} < \frac{\gamma^2\,h_*}{g+\gamma^2\,h_*}$$
and the function $h \mapsto \displaystyle\frac{\tanh(kh)}{kh}$ is
strictly decreasing on $(0,\infty)$ while $h \mapsto
\displaystyle\frac{\gamma^2\,h}{g\,+\,\gamma^2\,h}$ is strictly
increasing, we deduce that the nontrivial flows bifurcating from the
laminar flows at $m_-(h)$ contain stagnation points if and only if
$h>h^\ast$, where $h^\ast \in (0,h_*)$ is the unique solution to the
equation obtained by requiring equality in (\ref{stp}).

\medskip

The approach in \cite{W} is very different from ours, in that the
author performs a flattening $(X,Y) \mapsto
\left(X,\displaystyle\frac{Y}{\eta(X)}\right)$ of the free surface
$Y=\eta(X)$. (This method is therefore restricted to waves with
non-overhanging profiles.) The bifurcation-inducing laminar flows
identified there are exactly the same as ours. However, from the
family of laminar flows of any depth $h_0$, the local bifurcation
curve constructed in \cite{W} issuing from $ (h_0, m_\pm(h_0))$ is a
curve on ${\mathcal M}$ for which \be Q=2gh_0 + \Big(
\frac{m}{h_0}-\frac{\gamma h_0}{2}\Big)^2 \label{za}\ee for the
corresponding solutions of (\ref{apb}), but no information on the
conformal mean depth of the fluid domains is obtained (so that it
need not be $h_0$). There is a remark in \cite{W} that a
modification of the method there could have led to the existence of
a new curve on ${\mathcal M}$, for which, instead of (\ref{za}), the
condition \be h_0=\frac{k}{2\pi}\int_{-\pi/k}^{\pi/k} \eta(X)\,dX\ee
will hold for the corresponding solutions of (\ref{apb}).

The approach in \cite{CS}, while more general in that it covers also
non-constant vorticities, requires that there are no stagnation
points in the flow. More precisely, it relies on a hodograph
transform, for the existence of which it is essential that the
horizontal velocity never vanishes in the flow. In \cite{CS} the
relative mass flux $m$ is fixed and $h$ varies, so each nontrivial
bifurcation curve on $\mathcal{M}$ is contained in a hyperplane
$m=\text{constant}$. Since \cite{CS} only studies solutions for
which $\psi_Y<0$ in $\overline\Om$,  it only obtains $\lambda_-$ in
(\ref{lb7}) as a bifurcation point, and only in those cases for
which the associated laminar flow (\ref{lb8}) does not have
stagnation points. (However, the same method is easily seen to be
applicable also for solutions for which $\psi_Y>0$ in
$\overline\Om$, leading to $\lambda_+$ in (\ref{lb7}) as a
bifurcation point, though again only in those cases for which the
associated laminar flow (\ref{lb8}) does not have stagnation
points.)
 In \cite[Section 3]{CS} it was concluded that, for $k=1$, local bifurcation
from trivial flows always occurs for negative constant vorticity,
while in the case of positive constant vorticity there is a
necessary and
 sufficient condition on the value of the mass flux $m<0$, namely
\begin{equation}\label{lb15}
\tanh\Big(\sqrt{\frac{-2m}{\gamma}}\Big) >
\frac{-2\,m\gamma}{g+\gamma\sqrt{-2m\gamma}}.
\end{equation}
Let us check these results against ours. We are thus looking at the
laminar bifurcation-inducing flows for $k=1$, and with $\psi_Y<0$
everywhere in the closure of the fluid domain. Hence, by
(\ref{lb6}), (\ref{lb7}) and (\ref{wta}), and dropping the
subscripts,
\begin{equation}
\la=-\,\frac{\gamma\tanh(h)}{2} - \displaystyle\sqrt{\frac{
\gamma^2\tanh^2(h)}{4}+g\,\tanh(h)},\label{ab6}
\end{equation}
\begin{equation}\label{ab7}
m=\frac{\gamma h^2}{2}-\,\frac{\gamma h\tanh(h)}{2}-
h\,\displaystyle\sqrt{\frac{\gamma^2\tanh^2(h)}{4}+g\,{\tanh(h)}},
\end{equation}
\be \la+\gamma h<0.\label{an8}\ee We are interested in the range of
the mapping $h\mapsto m(h)$ given by (\ref{ab7}), restricted to the
set of $h\in (0,\infty)$ for which (\ref{an8}) holds. We will prove
that the range is $(-\infty, 0)$ if $\gamma<0$ and coincides with
the set for which (\ref{lb15}) holds, thus confirming the
conclusions of \cite{CS}. The case $\gamma<0$ is easy, so we
concentrate on the case $\gamma>0$.
 Recall that (\ref{ab6}) was
obtained from (\ref{bif}) with $n=1$ and $k=1$, hence
\be\la^2=(g-\gamma\lambda)\tanh(h),\label{an9}\ee as the solution
with \be \la<0,\label{an10}\ee while (\ref{ab7}) was obtained from
(\ref{ab6}) by, see (\ref{lb4}), \be m=\lambda h+\frac{\gamma
h^2}{2}.\label{an11}\ee Hence $m$ is in the required range if and
only if (\ref{an9}) and (\ref{an11}) are solvable with respect to
$\la\in\bdr$ and $h>0$ so that (\ref{an8}) and (\ref{an10}) are
satisfied. We solve first (\ref{an11}) with respect to $h$ for fixed
$\la$ satisfying (\ref{an10}). The condition \be \la\leq
-\sqrt{-2\gamma m}\label{an12}\ee is necessary and sufficient for
solvability, and the requirement (\ref{an8}) leads to the unique
solution
 \be
h=-\frac{\lambda}{\gamma}-\frac{\sqrt{\la^2+2\gamma
m}}{\gamma},\label{an15}\ee and to strict inequality in
(\ref{an12}). Plugging (\ref{an15}) into (\ref{an9}) we obtain
 \be
\tanh\left(-\frac{\lambda}{\gamma}-\frac{\sqrt{\la^2+2\gamma
m}}{\gamma}\right)= \frac{\la^2}{g-\gamma\la},\ee to be solved with
respect to $\lambda$ in the interval $(-\infty, -\sqrt{-2\gamma
m})$. The necessary and sufficient condition for this solvability is
easily seen to be exactly (\ref{lb15}), as required.

\begin{appendix}

\section{On conformal mappings for $L$-periodic strip-like domains}

We prove here the existence and uniqueness of the conformal mean
depth for any $L$-periodic strip-like domain $\Om$. For any such
domain $\Om$, let $\Om^R$ be the domain obtained from $\Om$ by
symmetrization with respect to the real axis $\mcb$, and let
$$\tilde\Omega=\Om\cup\mcb\cup\Om^R.$$
Let $H>0$ be such that $\tilde\Om$ is a subset of the horizontal
strip
$$S_H=\{ (x,y) \in \bdr^2:\ -H < y <H\}.$$
The application
$$\Phi(z)=\displaystyle\frac{e^{\frac{\pi}{2H}\,z}-1}{e^{\frac{\pi}{2H}\,z}+1}$$
maps $S_H$ conformally onto the unit disc $\bdd=\{\xi \in \bdc:\
|\xi|<1\}$. Let $\Pi=\Phi(\tilde\Om)$. The boundary of $\Pi$ is a
Jordan curve ${\mathcal J} \subset \bdd \cup \{\pm1\}$, symmetric
about the real axis and passing through the points $\pm1$. Clearly
$\Phi$ is a homeomorphism between the closures of $\tilde\Om$ and of
$\Pi$. Moreover, if $\bdd_+$ denotes the part of $\bdd$ situated in
the upper half-plane and $\Pi_+=\Pi\cap\bdd_+$, then $\Pi_+=
\Phi(\Om)$. Carath\'eodory's theorem \cite[Theorem 2.6 and Corollary
2.6, p.\ 24]{Po} ensures the existence and uniqueness of a conformal
map $\Phi_1$ from $\bdd$ onto $\Pi$, which has an extension as a
homeomorphism between the closures of these domains and is such that
$\Phi_1(\pm 1)=\pm 1$ and $\Phi_1(0)=0$. The uniqueness assertion
implies that $\Phi_1(\bar{z})=\overline{\Phi_1(z)}$ for all $z \in
\bdd$. This in turn implies that $\Phi_1$ maps the segment $[-1,1]
\subset \bdr$ onto itself, and the imaginary part of $\Phi_1$ does
not change sign in $\bdd_+$. Since the real part of $\Phi_1$ is
strictly increasing on the segment $[-1,1] \subset \bdr$, the
Cauchy-Riemann equations on that segment imply that the imaginary
part of $\Phi_1$ is positive in $\bdd_+$. Hence $\Phi_1$ maps
$\bdd_+$ onto $\Pi_+$. Denote $\Phi_0=\Phi^{-1}$. Since
\[\Phi_2(z)=\displaystyle\frac{e^{\frac{\pi}{2}\,(z+i)}-1}{e^{\frac{\pi}{2}\,(z+i)}+1}\]
maps $\mcr_2$ conformally onto $\bdd$ , we see that $\Phi_0 \circ
\Phi_1 \circ \Phi_2=U_0+iV_0$ with $U_0,\,V_0:\mcr_2 \to \bdr$, is a
conformal mapping of $\mcr_2$ onto $\tilde\Omega$, with an extension
as a homeomorphism between the closures of these domains, such that
 $\mcr_1$ is mapped onto $\Om$ and points symmetric in ${\mcr_2}$ with respect
to the line $y=-1$ are mapped into points in ${\tilde\Omega}$ that
are symmetric with respect to the real axis.

Let $\beta+i\delta$ be the inverse of $U_0+iV_0$, with
$\beta,\,\delta: \tilde\Om \to \bdr$. Since $\Om$ is an $L$-periodic
strip-like domain, we observe that the function
$$(X,Y) \mapsto \delta(X+L,Y)-\delta(X,Y)$$
is harmonic in $\Om$, and has a continuous extension to the closure
of $\Om$ which is zero on the boundary of $\Om$ since by
construction $\delta=-1$ on $\mcb$ and $\delta=0$ on $\mcs$. Being
also bounded in ${\Om}$ (since by construction $|\delta| \le 1$),
the Phragmen-Lindel\"of principle \cite{DM} ensures that this
function is identically zero, so that
\begin{equation}\label{de}
\delta(X+L,Y)=\delta(X,Y),\qquad (X,Y) \in \Om.
\end{equation}
The Cauchy-Riemann equations then imply that
\begin{equation}\label{be}
\beta(X+L,Y)-\beta(X,Y)=K,\qquad (X,Y) \in \Om,
\end{equation}
for some constant $K \in \bdr$. Note that necessarily $K>0$, since
the fact that the real part of $\Phi_1$ is a strictly increasing
homeomorphism of the segment $[-1,1]\subset\bdd$ implies that
$x\mapsto\beta(x,0)$ is a strictly increasing homeomorphism of
$\bdr$. From (\ref{de}) and (\ref{be}) we infer that
\begin{equation}\label{uv0}
\begin{cases}
U_0(x+K,y)=U_0(x,y)+L,\\
V_0(x+K,y)=V_0(x,y),
\end{cases}\qquad (x,y) \in {\mcr_1}.
\end{equation}
Replacing $ U_0+iV_0$ from $\mcr_1$ onto $\Om$ by $\tilde U+i\tilde
V: \mcr_h \to \Om$ given by \[\tilde U(x,y)= U_0(x/h,
y/h),\qquad\tilde V(x,y)= V_0(x/h, y/h), \qquad (x,y)\in\mcr_h,\]
where \[h={L}/{K},\] we see that $h$ satisfies the definition of the
conformal mean depth of $\Om$.

\bigskip\bigskip

\begin{center}
\includegraphics[width=12cm]{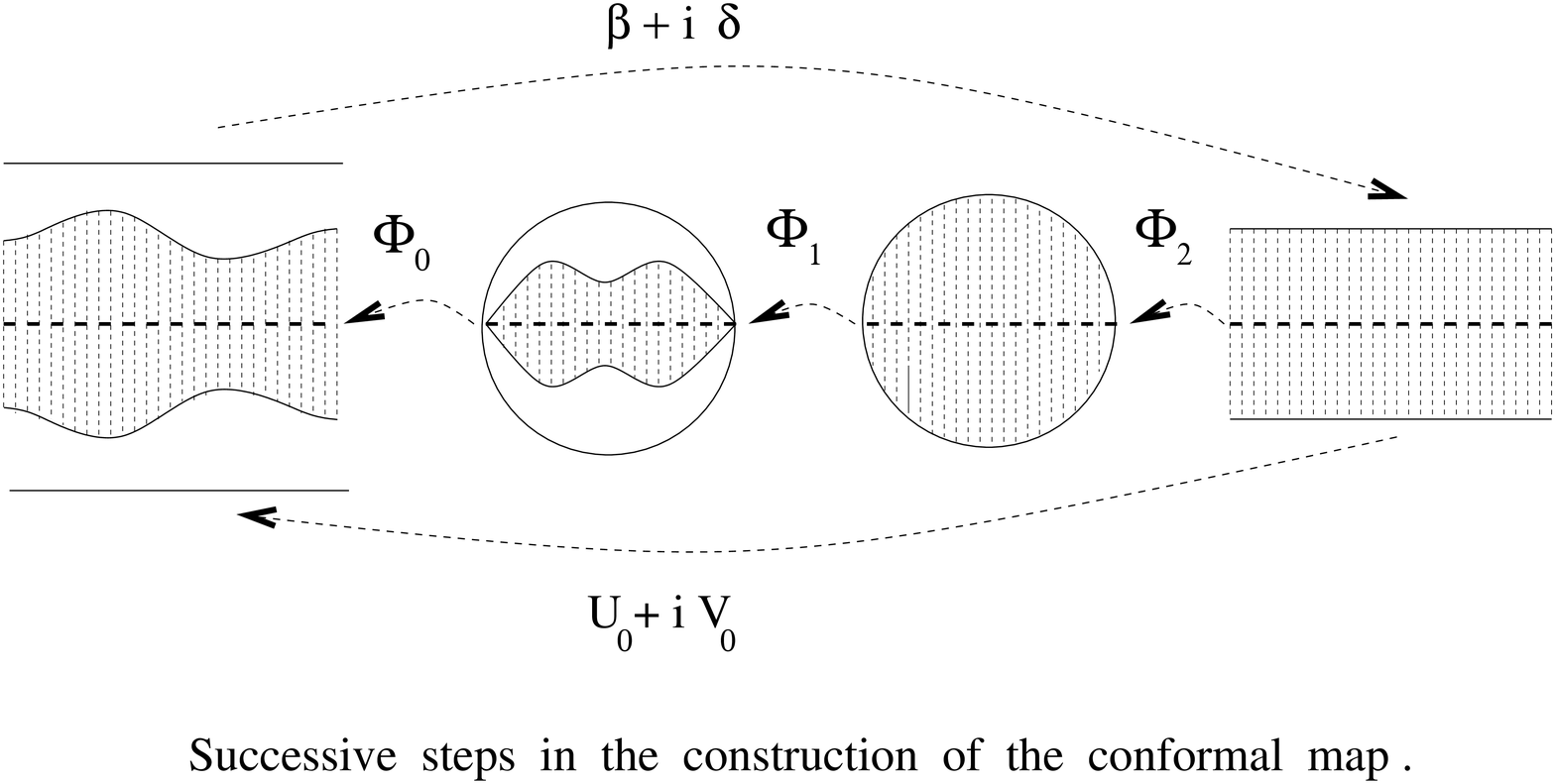}
\end{center}

\bigskip\bigskip

Suppose now that $\hat h $ also satisfies this definition, and let
$\hat U +i\hat V: \mcr_{\hat h}\to \Om $ be the corresponding
conformal mapping. Then
$$(\hat U+i\hat V)^{-1}\circ (\tilde U+i\tilde V)$$
is a conformal bijection from $\mcr_h$ onto $\mcr_{\hat h}$ and a
homeomorphism between the closures of these domains. Let us write
$$(\hat U+i\hat V)^{-1}\circ (\tilde U+i\tilde
V)=\tau+i\om,$$ where $\tau,\om:\mcr_h\to\bdr$ are harmonic
functions. It follows that \be\begin{cases}
\tau(x+L,y)=\tau(x,y)+L,\\
\omega(x+L,y)=\omega(x,y),
\end{cases}\qquad (x,y) \in {\mcr_h}.\label{to}\ee
Also,
\[\om (x,0)=0, \quad \om(x, -h)=-\hat h,\qquad x\in\bdr.\]
The maximum principle shows that
\[ \om (x,y)=\hat h y/h, \quad (x,y)\in\mcr_h,\]
which implies that
\[\tau(x,y)=(\hat h x/h)+C, \quad (x,y)\in\mcr_h,\]
where $C\in\bdr$ is a constant, a fact which, together with the
first relation in (\ref{to}) implies that $\hat h=h$.

If, in addition, $\mcs$ is a curve of class $C^{1,\alpha}$, then the
Jordan curve ${\mathcal J}$ is $C^{1,\alpha}_{\textnormal{loc}}$
everywhere with exception of the two points $\pm1$. By the
Kellogg-Warschawski theorem \cite[Theorems 3.5 and 3.6, p.\ 48]{Po},
the conformal map $\Phi_1$ inherits this smoothness up to the
boundary, namely it is of class $C^{1,\alpha}_{\textnormal{loc}}$ in
the closure of $\bdd$ with the exception of $\pm 1$ and, moreover,
its derivative is not zero at any point in $\overline\bdd\setminus
\{\pm 1$\}. Our construction ensures, taking into account the
periodicity, that $\tilde U,\,\tilde V \in
C^{1,\alpha}(\overline{\mcr_h})$ and
\[\tilde U_x^2+ \tilde V_x^2\neq 0\qquad \text{in }\overline{\mcr_h}.\]

We now show that $h$ coincides with the `mean depth' of a
$L$-periodic strip-like domain as defined in \cite{AT}. There, the
domain $\Om$ is considered by as being occupied by a fluid in steady
irrotational motion, so that $\mcs$ and $\mcb$ are streamlines for a
stream function $\psi$ which is $L$-periodic in $X$ throughout $\Om$
and satisfies (\ref{ap0}) with $\gamma=0$, (\ref{ap1}) and
(\ref{ap2}). Since the flow is irrotational, there exists an
associated velocity potential $\phi$ in $\Om$, which is a harmonic
function such that $\phi+i\psi$ is holomorphic. The velocity field
is then given by $(\psi_Y, -\psi_X)$, which also equals $(\phi_X,
\phi_Y)$ by the Cauchy-Riemann equations. The mean horizontal
velocity of the flow is then given by
\[c=\frac{1}{L}\int_X^{X+L} \phi_X(T, Y)\,dT= \frac{1}{L}(\phi(X+L,Y)-\phi(X,Y)),\]
a quantity which is easily seen not to depend on $(X,Y)$ as long as
the horizontal line segment joining $(X,Y)$ and $(X+L,Y)$ is
contained in $\Om$. The `mean depth' is then defined as $m/c$, the
ratio between the flux and the mean horizontal velocity. With
$\beta,\delta$ as previously defined in Appendix A, note that the
maximum principle implies that
\[\psi(X,Y)=m\delta(X,Y), \quad (X,Y)\in\Om, \]
and then the Cauchy-Riemann equations show that \[
\phi(X+L,Y)-\phi(X,Y)=m(\beta(X+L,Y)-\beta(X,Y)), \quad (X,Y)\in
\Om.\] Therefore
\[cL=mK\]
and hence
\[h=m/c.\]
Consequently the conformal mean depth of $\Om$ coincides with the
`mean depth' defined in \cite{AT}.

\section{A commutator estimate}

\begin{lemma}\label{lm}
  If $f \in C^{1,\alpha}_{2\pi}$ and $g\in C^{0,\alpha}_{2\pi}$ with  $\alpha \in
(0,1)$, then \[f\,\mcc(g)-\mcc(fg) \in
C^{1,\delta}_{2\pi}\quad\text{ for all }\delta \in (0,\alpha).\]
\end{lemma}

\begin{proof}
Let \[\Theta(t)=f(t)\,\mcc(g)(t)-\mcc(fg)(t)\quad\text{for all
}t\in\bdr.\] Using the definition of $\mcc$ we obtain that, for all
$t\in\bdr$, \be \Theta(t)= \frac{1}{2\pi}\int_{-\pi}^\pi
\cot\left(\frac{t-s}{2}\right)\,\{f(t)-f(s)\}\, g(s)\,ds,\ee where,
since $f \in C^{1,\alpha}_{2\pi}$, the integral is not a principal
value integral, but a genuine Lebesgue integral. In proving that
$\Theta\in C^{1,\delta}_{2\pi}$ for all $\delta \in (0,\alpha)$ we
will make use of the following inequalities: \be |x-\sin(x)|\leq
\frac{1}{6}|x|^3\quad\text{for all $x\in\bdr$,}\label{yr}\ee \be
\frac{1}{\sin^2(x)}\leq \sigma (|x|)\frac{1}{x^2}\quad\text{for all
$x$ with } 0<|x|<\pi,\label{yt} \ee \be
\left|\cot(x)-\frac{1}{x}\right|\leq \varsigma(|x|)|x| \quad
\text{for all $x$ with }0<|x|<\pi,\label{ys}\ee  where
$\sigma,\varsigma:(0,\pi)\to[0,\infty)$ are bounded functions on
$(0,a]$ for each $a\in(0,\pi)$.

We show first that $\Theta$ is differentiable on $\bdr$ and its
derivative is given, for all $t\in\bdr$, by the formula obtained by
formal differentiation under the integral sign: \be
\Theta'(t)=\frac{1}{4\pi}\int_{-\pi}^{\pi}\frac{ f'(t) \sin
(t-s)-\{f(t)-f(s)\}}{\sin^2\left(\frac{t-s}{2}\right)}g(s)\,ds.\label{deriv}\ee
Note that the above integral finite. Indeed, the integral
\be\frac{1}{4\pi}\int_{t-\pi}^{t+\pi}\frac{ f'(t) \sin
(t-s)-\{f(t)-f(s)\}}{\sin^2\left(\frac{t-s}{2}\right)}g(s)\,ds,\label{pel}\ee
is finite, as it can be seen upon using (\ref{yr}), (\ref{yt}) and
the Mean Value Theorem, by which \begin{align}
|f'(t)(t-s)-\{f(t)-f(s)\}|&=|t-s|\,|f'(t)-f'(\xi)|\non\\ &\leq
C|t-s|^{1+\alpha},\non\end{align}
 for some $\xi=\xi(t,s)$ with
$|\xi-t|\leq |s-t|$. Hence the integral in (\ref{deriv}) is finite,
being obtained from that in (\ref{pel}) by a rearrangement using the
periodicity of the functions involved. In the above and throughout
what follows, we denote by $C$ a constant which depends only on the
norm of $f$ in $C^{1,\alpha}_{2\pi}$ and the norm of $g$ in
$C^{0,\alpha}_{2\pi}$, and is allowed to vary from line to line.

Since the operator defining $\Theta$ commutes with translations, it
suffices to show that $\Theta$ is differentiable at $t=0$, with
derivative at $t=0$ given by (\ref{deriv}). Note that, for any $t$
close to $0$, one can write
$$2\pi\,\displaystyle\frac{\Theta(t)-\Theta(0)}{t}= I_1(t)+I_2(t),$$
where
$$\begin{array}{c}
I_1(t)=\displaystyle\frac{1}{t}\int_{-2|t|}^{2|t|}
\left\{\cot\left(\frac{t-s}{2}\right)\{f(t)-f(s)\}-\displaystyle\cot\left(\frac{-s}{2}\right)\{f(0)-f(s)\}\right\}\,
g(s)\,ds,\\[0.4cm]
I_2(t)=\,\displaystyle\int_{(-\pi,-2|t|)\cup(2|t|,\pi)}
\frac{1}{t}\left\{\cot\left(\frac{t-s}{2}\right)\{f(t)-f(s)\}\,
 -\,\displaystyle\cot\left(\frac{-s}{2}\right)\{f(0)-f(s)\}\right\}g(s)\,ds.\\[0.4cm]
\end{array}$$
Using (\ref{ys}), one can write $I_1(t)$ as the sum of a term which
tends to $0$ as $t\to 0$ and
\[\frac{2}{t}\int_{-2|t|}^{2|t|}
\left\{\frac{f(t)-f(s)}{t-s}- \frac{f(0)-f(s)}{0-s}\right\}
g(s)\,ds\] Since, for fixed $t$, the integrand in the above is
continuous as a function of $s$, the Mean Value Theorem for
integrals shows that the above quantity equals
\[\frac{8\,|t|}{t}\, \left\{\frac{f(t)-f(\xi)}{t-\xi}- \frac{f(0)-f(\xi)}{0-\xi}\right\}
g(\xi)\] for some $\xi\in(-2|t|, 2|t|)$. Using again the Mean Value
Theorem, the above quantity equals
\[\frac{8\,|t|}{t}\, \left\{f'(\xi_1)- f'(\xi_2)\right\}
g(\xi)\] for some $\xi_1,\xi_2\in(-2|t|, 2|t|)$. Since $f'$ is
continuous and $g$ is bounded, the above quantity tends to $0$ as
$t\to 0$. We have thus proved that \be I_1(t)\to 0\quad\text{as
}t\to 0.\label{vas}\ee We now prove that, as $t\to 0$, \be I_2(t)\to
\int_{-\pi}^{\pi}\frac{ f'(0) \sin
(-s)-\{f(0)-f(s)\}}{2\sin^2\left(\frac{s}{2}\right)}\,g(s)\,ds.\label{mpn}\ee
For this, we use the Dominated Convergence Theorem.  Observe that
 $I_2(t)$ can be written as an integral over $(-\pi,\pi)$, with the
 integrand given by
$$\chi_{(-\pi,-2|t|)\cup(2|t|,\pi)}\,\Big(\frac{f(t)-f(0)}{t}
\,\cot(\tfrac{t-s}{2})\,+\,
\{f(0)-f(s)\}\,\frac{\cot(\frac{t-s}{2})\,-\,\cot(-\,\frac{s}{2})}{t}
\Big)\,g(s)$$ where $\chi_A$ stands for the characteristic function
of the set $A$. Firstly, note that, as $t\to 0$, the above integrand
converges pointwise to the integrand on the right-hand side of
(\ref{mpn}). Secondly, note that, for each $s\in (-\pi,0)\cup
(0,\pi)$ and for each $t$ with $0<|t|< |s|/2$, the Mean Value
Theorem yields that \be
\begin{array}{c}
\displaystyle\frac{1}{t}\,\Big\{\cot\left(\frac{t-s}{2}\right)\,\{f(t)-f(s)\}\,-\,
\cot\left(\frac{-s}{2}\right)\,\{f(0)-f(s)\}\Big\}\\[0.4cm]
=\displaystyle\frac{ f'(\xi)
\sin(\xi-s)\,-\,\{f(\xi)-f(s)\}}{2\sin^2\left(\frac{\xi-s}{2}\right)}
\end{array}
\ee for some $\xi$ between $0$ and $t$. Using again (\ref{yr}) and
(\ref{yt}), we obtain \begin{align} \left|\frac{ f'(\xi) \sin
(\xi-s)-\{f(\xi)-f(s)\}}{2\sin^2\left(\frac{\xi-s}{2}\right)}\right|
&\leq C\left|\frac{ f'(\xi)
(\xi-s)-\{f(\xi)-f(s)\}}{(\xi-s)^2}\right|+C|\xi-s|\non\\
&=C\left|\frac{f'(\xi)-f'(\xi_0)}{\xi-s}\right|+C|\xi-s|\non\\
&\leq C|\xi-s|^{\alpha-1}\non\\
&\leq C|s|^{\alpha-1},\non
\end{align}
where $\xi$ is between $0$ and $t$, $\xi_0$ is between $\xi$ and
$s$, and we took into account that $0<|t|< |s|/2$ and $f\in
C^{1,\alpha}_{2\pi}$. Since $s\mapsto |s|^{\alpha -1}$ is integrable
and $g$ is bounded, we have thus checked the conditions of the
Dominated Convergence Theorem. Therefore, (\ref{mpn}) holds and,
since (\ref{vas}) holds, it follows  that $\Theta$ is differentiable
at $t=0$, with derivative given by (\ref{deriv}) evaluated at $t=0$.
As noted earlier, this implies that $\Theta$ is differentiable on
$\bdr$, with derivative given by (\ref{deriv}).

We now prove that $\Theta'\in C^{0,\delta}_{2\pi}$ for each
$\delta\in (0,\alpha)$. Since $\Theta'$ commutes with translation,
it suffices to show that \be |\Theta'(t)-\Theta'(0)|\leq
C|t|^\delta\quad\text{for all $t$ close to $0$}.\label{hold} \ee The
change of variables $s\mapsto (t-s)$ in (\ref{deriv}) leads to the
following formula \be
\Theta'(t)=\frac{1}{4\pi}\int_{-\pi}^{\pi}\frac{ f'(t-s) \sin
(s)\,-\,\{f(t)-f(t-s)\}}{\sin^2\left(\frac{s}{2}\right)}\,g(t-s)\,ds,\label{dva}\ee
for all $t\in\bdr$. It follows that, for each $t$ close to $0$, one
can write \be
4\pi\,\Big(\Theta'(t)-\Theta'(0)\Big)=J_1(t)+J_2(t)+J_3(t),\label{dete}\ee
where
$$J_1(t)=\int_{-2|t|}^{2|t|}\frac{ f'(t)\sin
(s)-\{f(t)-f(t-s)\}}{\sin^2\left(\frac{s}{2}\right)}\,g(t-s)\,ds$$
$$J_2(t)=-\,\int_{-2|t|}^{2|t|}\,\frac{
f'(0) \sin
(s)-\{f(0)-f(-s)\}}{\sin^2\left(\frac{s}{2}\right)}\,g(-s)\,ds$$
$$\begin{array}{ccc}
J_3(t)
&=&\displaystyle\int_{(-\pi,-2|t|)\cup(2|t|,\pi)}\,\Big\{\frac{
f'(t) \sin
(s)-\{f(t)-f(t-s)\}}{\sin^2\left(\frac{s}{2}\right)}\,g(t-s)\\[0.5cm]
&& -\,\displaystyle\frac{ f'(0)
\sin(s)-\{f(0)-f(-s)\}}{\sin^2\left(\frac{s}{2}\right)}\,g(-s)\Big\}\,ds.
\end{array}$$
 Note first that
\begin{align}\label{opw}
|J_1(t)|&\leq\int_{-2|t|}^{2|t|} |f'(t)||g(t-s)|\frac{ |\sin
(s)-s|}{\sin^2\left(\frac{s}{2}\right)}+\int_{-2|t|}^{2|t|}
|g(t-s)|\frac{
\left|\int_{-s}^{0}\{f'(t)-f'(t+\tau)\}\,d\tau\right|}
{\sin^2\left(\frac{s}{2}\right)}\,ds\\&\leq C|t|^2+C|t|^\alpha\non\\
&\leq C|t|^\alpha,\non
\end{align}
where we have used (\ref{yt}) and (\ref{yr}), the boundedness of
$f'$ and $g$, and the fact that $f'\in C^{0,\alpha}_{2\pi}$. Then,
exactly the same argument gives that \be |J_2(t)|\leq
C|t|^\alpha.\ee Also, one can write \be J_3(t)=K_1(t)+K_2(t),\ee
where
\begin{align}
K_1(t)&=\int_{(-\pi,-2|t|)\cup(2|t|,\pi)}\frac{ \{f'(t)-f'(0)\} \sin
(s)+\{f(0)-f(-s)\}-\{f(t)-f(t-s)\}}{\sin^2\left(\frac{s}{2}\right)}\,g(t-s)\,ds\non\\
K_2(t)&=\int_{(-\pi,-2|t|)\cup(2|t|,\pi)}\frac{ f'(0) \sin
(s)-\{f(0)-f(-s)\}}{\sin^2\left(\frac{s}{2}\right)}\,\big\{g(t-s)-g(-s)\big\}\,ds.\non
\end{align}
Using (\ref{yt}), the boundedness of $g$ and the fact that $f'\in
C^{0,\alpha}_{2\pi}$, it follows that
\begin{align}\label{an1}|K_1(t)|&\leq
C\int_{(-\pi,-2|t|)\cup(2|t|,\pi)}|g(t-s)|\Big\{\frac{|f'(t)-f'(0)|}{|s|}+
\frac{\left|\int_{-s}^{0}(f'(\tau)-f'(t+\tau))\,d\tau\right|}{|s|^2}\Big\}\,ds\\
&\leq C|t|^\alpha\int_{(-\pi,-2|t|)\cup(2|t|,\pi)}\frac{1}{|s|}\,ds\non\\
&\leq\non C|t|^\alpha|\log|t||\non\\&\leq
C|t|^\delta.\non\end{align} Also, using (\ref{yt}), the boundedness
of $f'$ and the the fact that $g\in C^{0,\alpha}_{2\pi}$, it follows
that
\begin{align}\label{an2}|K_2(t)|&\leq
C\int_{(-\pi,-2|t|)\cup(2|t|,\pi)}|g(t-s)-g(-s)|\Big\{\frac{
|f'(0)|}{|s|}
+\frac{\left|\int_{-s}^{0}f'(\tau)\,d\tau\right|}{|s|^2}\Big\}\,ds\\
&\leq C|t|^\alpha\int_{(-\pi,-2|t|)\cup(2|t|,\pi)}\frac{1}{|s|}\,ds\non\\
&\leq C|t|^\alpha|\log|t||\non\\
&\leq C|t|^\delta.\non\end{align} Combining the estimates
(\ref{dete})--(\ref{an2}) yields (\ref{hold}). As noted earlier,
this implies that $\Theta'\in C^{0,\delta}_{2\pi}$. This completes
the proof.

\end{proof}

\end{appendix}

 {\small
{\sc Acknowledgement.} The support of the Vienna Science and
Technology Fund (WWTF) is gratefully acknowledged. The authors thank
the referee for a useful suggestion concerning the regularity result
in Section 4.}

\providecommand{\href}[2]{#2}

\end{document}